\documentclass{eptcs-modified}

\usepackage{amsmath}
\usepackage{amssymb}
\usepackage{url}
\usepackage{doi}
\usepackage{amsthm}
\usepackage{graphicx}
\usepackage{amscd}
\usepackage[usenames,dvipsnames,svgnames,table]{xcolor}

\usepackage{todonotes}
\setuptodonotes{inline}

\usepackage{ifthen}
\usepackage{xargs}

\usepackage{float}
\usepackage{subcaption}
\usepackage{tikzit}
\tikzstyle{box}=[fill={rgb,255: red,228; green,228; blue,228}, draw=black, shape=rectangle]
\tikzstyle{blankbox}=[shape=rectangle, draw=white, fill=white]
\tikzstyle{pnode}=[fill=black, draw=black, inner sep=2pt, shape=circle]
\tikzstyle{extnode}=[fill=gray, draw=black, inner sep=2pt, shape=circle]

\tikzstyle{reducible}=[->, dashed, line width=1pt]
\tikzstyle{strictreducible}=[->, line width=1pt]
\tikzstyle{nonreducible}=[->, draw=red, line width=1pt]
\tikzstyle{incomp}=[draw=red, <->, line width=1pt]
\tikzstyle{arith_sep}=[-, draw=blue]
\tikzstyle{ngroup}=[ draw=gray, thick, rounded corners, fill=gray!20, minimum height=1em] 
\usetikzlibrary{backgrounds}
\usetikzlibrary{arrows}
\usetikzlibrary{shapes,shapes.geometric,shapes.misc}
\usetikzlibrary{matrix,shapes,positioning,intersections,quotes}
\usetikzlibrary {decorations.pathmorphing, decorations.text}
\usetikzlibrary{decorations.pathreplacing,calligraphy}

\usetikzlibrary{calc}
\usetikzlibrary{fit}

\newcommand\arrowchain[4]{%1:A, 2:B, 3:pos1, 4:pos2
    \tikzset{nd/.style={minimum size=0pt, inner sep=0pt}}%
    \draw[draw=none] (#1) -- node[nd,pos=#3] (C) {} node[nd,pos=#4] (D) {} (#2);%
    \draw[dotted]  (C) -- (D);%
    \draw[->]        (#1) -- (C);%
    \draw[->]        (D) -- (#2);%
}

\include{thref}

\makeatletter
\providecommand{\bigsqcap}{%
\mathop{%
\mathpalette\@updown\bigsqcup
}%
}
\newcommand*{\@updown}[2]{%
\rotatebox[origin=c]{180}{$\m@th#1#2$}%
}
\makeatother

\theoremstyle{definition}
\newtheorem{theorem}{Theorem}[section]
\newtheorem{definition}[theorem]{Definition}

\newtheorem{corollary}[theorem]{Corollary}
\newtheorem{proposition}[theorem]{Proposition}

\newtheorem{lemma}[theorem]{Lemma}

\newtheorem{question}[theorem]{Open Question}

\newcommand{\Cantor}{{2^\mathbb{N}}}

\newcommand{\id}{\textrm{id}}
\newcommand{\dom}{\operatorname{dom}}
\newcommand{\ran}{\operatorname{ran}}

\newcommand{\Baire}{\mathbb{N}^\mathbb{N}}
\newcommand{\hide}[1]{}

\newcommand{\C}{\mathsf{C}}

\newcommand{\lpo}{\mathsf{LPO}}

\newcommand{\leqW}{\leq_{\textrm{W}}}
\newcommand{\leW}{<_{\textrm{W}}}
\newcommand{\equivW}{\equiv_{\textrm{W}}}

\newcommand{\geqW}{\geq_{\textrm{W}}}

\newcommand{\nleqW}{\nleq_\textrm{W}}
\newcommand{\leqsW}{\leq_{\textrm{sW}}}
\newcommand{\equivsW}{\equiv_{\textrm{sW}}}
\newcommand{\Sort}{\operatorname{Sort}}

\newcommand{\DS}{\mathsf{DS}}
\newcommand{\BS}{\mathsf{BS}}
\newcommand{\DSfe}{\mathsf{DS}^{\mathsf{FE}}}
\newcommand{\BSfe}{\mathsf{BS}^{\mathsf{FE}}}
\newcommand{\KL}{\mathsf{KL}}

\newcommand{\NON}{\mathsf{NON}}
\newcommand{\ACC}{\mathsf{ACC}}
\newcommand{\BStree}{\BS|_{\operatorname{Tree}}}
\newcommand{\ExtVer}{\operatorname{ExtVer}}

\newcommand{\pitaccN}{\boldsymbol{\Pi}^0_2\mathsf{-ACC}_\mathbb{N}}

\newcommand{\pitacc}[1]{\boldsymbol{\Pi}^0_2\mathsf{-ACC}_{#1}}

\makeatletter

\def\function{\@ifstar\@function\@@function}
\newcommand{\@function}[2]{#1 \to #2}
\newcommand{\@@function}[2]{\colon #1 \to #2}

\def\mfunction{\@ifstar\@mfunction\@@mfunction}
\newcommand{\@mfunction}[2]{#1 \rightrightarrows #2}
\newcommand{\@@mfunction}[2]{\colon #1 \rightrightarrows #2}

\makeatother
\newcommand{\pfunction}[2]{:\subseteq #1 \to #2}
\newcommand{\pmfunction}[2]{:\subseteq #1 \rightrightarrows #2}

\newcommand{\st}{{}\,:\,{}}
\newcommand{\defiff}{:\hspace{-1mm}\iff}
\newcommand{\length}[1]{|{#1}|}
\newcommand{\concat}{\smash{\raisebox{.9ex}{\ensuremath\smallfrown} }}

\newcommand{\restrict}[1]{\ensuremath{\left. \hspace{-1mm} \right|_{#1}}}

\newcommand{\parallelization}[1]{\widehat{#1}}
\newcommand{\compproduct}{*}

\newcommand{\sequence}[2]{(#1)_{{#2}}}
\newcommand{\pairing}[1]{\langle #1 \rangle}

\newcommand{\DetPartX }[2]{\mathrm{Det}_{ \mathbf{#1} }(#2)}
\newcommand{\PiBound}{\boldsymbol{\Pi}^1_1\mathsf{-Bound}}

\newcommand{\RT}[2]{\mathsf{RT}^{{#1}}_{{#2}}}

\newcommand{\Choice}[1]{\mathsf{C}_{#1}}

\newcommand{\CNatural}{\mathsf{C}_{\mathbb{N}}}
\newcommand{\UCNatural}{\mathsf{UC}_{\mathbb{N}}}
\newcommand{\turingreducible}{\le_\mathrm{T}}
\newcommand{\weireducible}{\leqW}
\newcommand{\weiequiv}{\equivW}
\newcommand{\strongweireducible}{\leqsW}

\newcommand{\Det}[1]{\operatorname{Det}\!\left(#1\right)}
\newcommand{\fop}[1]{{}^1#1}
\newcommand{\Fin}[2]{\operatorname{Fin}_{#1}\!\left(#2\right)}

\newcommand{\cantor}{2^{<\mathbb{N}}}
\newcommand{\baire}{\mathbb{N}^{<\mathbb{N}}}
\newcommand{\prefix}{\sqsubseteq}
\newcommand{\pprefix}{\sqsubset}

\newcommand{\extend}{\sqsupseteq}

\newcommand{\mflim}{\mathsf{lim}}

\newcommand{\boldfaceDelta}{\boldsymbol{\Delta}}
\newcommand{\boldfaceSigma}{\boldsymbol{\Sigma}}
\newcommand{\boldfacePi}{\boldsymbol{\Pi}}
\newcommand{\boldfaceGamma}{\boldsymbol{\Gamma}}
\newcommand{\SDDCC}{\boldfaceSigma^1_1\mathsf{-DUCC}}

\newcommand{\codedBS}[1]{\ifthenelse{\equal{#1}{}}{\mathsf{BS}}{ {#1}\text{-}\mathsf{BS} }}
\newcommand{\codedDS}[1]{\ifthenelse{\equal{#1}{}}{\mathsf{DS}}{ {#1}\text{-}\mathsf{DS} }}
\makeatletter
\def\blfootnote{\xdef\@thefnmark{}\@footnotetext}
\makeatother

\begin{document}

\title{The weakness of finding descending sequences in ill-founded linear orders}

\author{
Jun Le Goh
\institute{National University of Singapore, Singapore}
\email{gohjunle@nus.edu.sg}
\and
Arno Pauly
\institute{Department of Computer Science\\Swansea University, Swansea, UK\\}
\email{Arno.M.Pauly@gmail.com}
\and
Manlio Valenti
\institute{Department of Computer Science\\Swansea University, Swansea, UK\\}
\email{manliovalenti@gmail.com}
}

\def\titlerunning{The weakness of $\DS$}
\def\authorrunning{J.L.~Goh, A.~Pauly, M.~Valenti}
\maketitle

\vspace*{-3ex}

\begin{abstract}
	We explore the Weihrauch degree of the problems ``find a bad sequence in a non-well quasi order'' ($\BS$) and ``find a descending sequence in an ill-founded linear order'' ($\DS$). We prove that $\DS$ is strictly Weihrauch reducible to $\BS$, correcting our mistaken claim in \cite{goh-pauly-valenti}. This is done by separating their respective first-order parts. On the other hand, we show that $\BS$ and $\DS$ have the same finitary and deterministic parts, confirming that $\BS$ and $\DS$ have very similar uniform computational strength. We prove that K\"onig's lemma $\KL$ and the problem $\mathsf{wList}_{\Cantor,\leq\omega}$ of enumerating a given non-empty countable closed subset of $\Cantor$ are not Weihrauch reducible to $\DS$ or $\BS$, resolving two main open questions raised in \cite{goh-pauly-valenti}. We also answer the question, raised in \cite{dghpp20}, on the existence of a ``parallel quotient'' operator, and study the behavior of $\BS$ and $\DS$ under the quotient with some known problems.
\end{abstract}


\blfootnote{We are grateful to Takayuki Kihara for pointing out the mistake in our previous article. We also thank C\'ecilia Pradic for comments which greatly improved the presentation of some results. We thank the anonymous reviewers for their careful  reading of the paper. This paper extends the conference paper \cite{gohpaulyvalenti2-cie}.}

\vspace*{-3ex}

\tableofcontents
\section{Introduction}

Linear and quasi-orders are ubiquitous structures that play an important role in all areas of mathematics. Computability theory has been successfully applied to highlight the difference between classical and effective properties of orders. A well-known example is that the simple fact that every infinite linear order contains an infinite ascending sequence or an infinite descending sequence is not computably true. A more extreme example is the existence of computable ill-founded linear orders with no hyperarithmetic descending sequence (see e.g.\ \cite[Lemma III.2.1]{SacksHRT}). These orders have been extensively used and studied in reverse mathematics (under the name of pseudo-well-orders) \cite{Simpson09}. We refer the reader to \cite{Downey98} for a more comprehensive presentation of computability-theoretical results on linear and partial orders. 

The natural generalization of well-orders in the context of quasi-orders is the notion of \emph{well quasi-orders}: formally, a quasi-order $(Q,\preceq)$ is called well quasi-order (abbreviated wqo) if, for every infinite sequence $\sequence{q_n}{n\in\mathbb{N}}$ of elements of $Q$, there are $i,j$ with $i<j$ such that $q_i\preceq_Q q_j$. This can be restated by saying that a quasi-order is a wqo if it contains no infinite bad sequences, where a (possibly finite) sequence $\sequence{q_n}{n}$ is called bad if $q_i \not\preceq_Q q_j$ for every $i<j$. Equivalently, wqo's can be defined as quasi-orders that contain no infinite descending sequence and no infinite antichain. There is an extensive literature on the theory of wqo's. For an overview, we refer the reader to \cite{MarconeWQO}.

\smallskip

In this paper, we continue our investigation (started in \cite{goh-pauly-valenti}) of the uniform computational strength of the problems:
\begin{itemize}
	\item given a countable ill-founded linear order, find an infinite Descending Sequence in it ($\DS$), and
	\item given a countable non-well quasi-order, find a Bad Sequence in it ($\BS$).
\end{itemize}
A suitable framework for this is given by Weihrauch reducibility (see \cite{bgp21} for a self-contained introduction). Several results on the Weihrauch degree of $\DS$ were proved in our previous work, where we showed e.g.\ that, despite the fact that $\DS$ is not a hyperarithmetic problem, it is rather weak from the uniform point of view, as it only computes the limit computable functions. In the language of Weihrauch reducibility, this can be stated as $\Det{\DS}\weiequiv\mflim$. We also explored how the uniform computational strength changes when working with $\boldfaceGamma$-presented orders, where $\boldfaceGamma$ is a Borel pointclass, or $\boldfaceDelta^1_1$, $\boldfaceSigma^1_1$, or $\boldfacePi^1_1$. Other results on the Weihrauch degree of principles related to well-orders and wqo are obtained in \cite{carluccimainardizdanowski}. The computational strength of descending sequences was also independently explored in \cite[\S 4]{calvertfranklinturetsky}.

This paper is organized as follows: in Section~\ref{sec:background}, we introduce the notation and provide a brief overview of the necessary background notions on Weihrauch reducibility. 

In Section~\ref{sec:separation}, we refute the claim, falsely stated in \cite[Proposition 4.5]{goh-pauly-valenti}, that $\DS\weiequiv\BS$. In fact, we obtain the separation $\DS\leW\BS$ (\thref{theo:bsdssep}) by proving a separation between the respective first-order parts.

The following is a list of the results in \cite{goh-pauly-valenti} that are affected:
\begin{itemize}
	\item \cite[Proposition 4.5]{goh-pauly-valenti};
	\item \cite[Corollary 5.4]{goh-pauly-valenti} and \cite[Corollary 5.16]{goh-pauly-valenti}: these are one-line relativizations of \cite[Proposition 4.5]{goh-pauly-valenti};
	\item \cite[Corollary 5.13]{goh-pauly-valenti}: the equivalences $\codedDS{\boldfacePi^0_k} \weiequiv \codedDS{\boldfaceDelta^0_{k+1}} \weiequiv \codedDS{\boldfaceSigma^0_{k+1}}$ are unaffected, but the reductions involving $\codedBS{\boldfacePi^0_k}$ and $\codedBS{\boldfaceDelta^0_{k+1}}$ were obtained using \cite[Corollary 5.4]{goh-pauly-valenti} and transitivity.
\end{itemize}

We are not aware whether all the above-mentioned claims admit a counterexample. 

All the other results in \cite{goh-pauly-valenti} do not use \cite[Proposition 4.5]{goh-pauly-valenti}; they either deal only with $\DS$ or are standalone results about $\BS$ which are not affected by the above error. In particular, \cite[Theorem 5.3]{goh-pauly-valenti}, \cite[Theorem 5.14]{goh-pauly-valenti}, \cite[Proposition 5.15]{goh-pauly-valenti}, \cite[Theorem 5.23]{goh-pauly-valenti} and \cite[Corollary 5.24]{goh-pauly-valenti} are correct to the best of our knowledge. This list of errata is also available in the arxiv version of \cite{goh-pauly-valenti}.

The rest of the paper is devoted to better understanding the Weihrauch degrees of $\DS$ and $\BS$. More precisely, in Section~\ref{sec:klbs}, we answer (negatively) two main open questions raised in \cite[Questions 6.1 and 6.2]{goh-pauly-valenti}, namely whether $\KL$ and $\mathsf{wList}_{2^{\mathbb{N}},\leq\omega}$ are Weihrauch reducible to $\DS$ (Corollaries \ref{cor:KL_DS} and \ref{cor:wList_DS}). Both results follow from a more general characterization (\thref{theo:dsproducts}), stating that 
\[  \mflim \equivW \max_{\leqW}\{f \st \parallelization{\ACC_\mathbb{N}} \times f \leqW \DS\}\equivW \max_{\leqW}\{f \st \parallelization{\ACC_\mathbb{N}} \times f \leqW \BS\}. \] 
In other words, even though $\parallelization{\ACC_\mathbb{N}}$ is fairly weak
(in particular it is below $\mflim$, $\KL$ and $\DS$), neither $\DS$ nor 
$\BS$ can compute $\parallelization{\ACC_\mathbb{N}} \times f$ if $f \nleqW \mflim$.

In Section~\ref{sec:fin_BS_DS}, we show that, despite $\DS$ and $\BS$ not being Weihrauch equivalent, their finitary and deterministic parts are in fact the same (\thref{theo:finkbs} and \thref{cor:det(bs)}). In other words, it is necessary to consider a non-finitary, non-deterministic problem (such as $\fop{\BS}$) in order to separate $\BS$ and $\DS$.

Section~\ref{sec:godel} contains a short detour where we answer another question that was left open in our previous paper (\cite[Question 3.5]{goh-pauly-valenti}), namely whether the first-order part and the deterministic part commute. We answer this question by constructing an explicit counterexample for which the two operations do not commute (\thref{thm:fop_Det_below_Det_fop}). 

In Sections~\ref{sec:pitaccn}~and~\ref{sec:antichain}, we analyze the following question: Is the restriction of $\BS$ to trees stronger than $\DS$? This question is very natural when considering that, when proving that $\BS\not\weireducible \DS$, we construct a ``tree-like structure'' (a partial order admitting a tree decomposition, see \thref{def:treedecomposition}). In Section~\ref{sec:pitaccn} we obtain some technical results on the degrees of the problems $\pitacc{k}$ and $\pitaccN$. In Section~\ref{sec:antichain}, we show that the restriction $\BStree$ of $\BS$ to trees with infinite width is much weaker than $\BS$ by proving that $\Fin{}{\BStree} \equivW \Det{\BStree} \equivW \id$ (\thref{cor:fin_det_bstree}) and that $\fop{\BStree}$ is equivalent to the problem that maps a tree with infinite width to some $v\in T$ that belongs to some infinite antichain in $T$ (\thref{thm:fop_bstree}). 

Finally, in Section~\ref{sec:quotient} we observe that one of our main technical tools (\thref{theo:dsproducts}) provides an example of a ``parallel quotient'' \cite[Remark 1]{dghpp20}, namely a specific case where 
\[ \max_{\leqW} \{h \st h \times g \leqW f\} \]
is defined. Whether this $\max$ exists for all choices of $f$ and $g$ was asked in \cite[Remark 3.11]{dghpp20}. We answer their question by showing that this operator is always defined when $g\neq\emptyset$ (\thref{def:quotient} and \thref{thm:quotient_total}). We conclude the paper by analyzing the behavior of $\BS$ and $\DS$ under the quotient with some known problems.

\section{Background}
\label{sec:background}

We now briefly introduce the relevant notions in Weihrauch complexity. For a more thorough presentation, we refer the reader to \cite{bgp21}. We end this section with the observation that $\BS$ is equivalent to its restriction to partial orders (\thref{prop:BS_qo_po}).

A \emph{represented space} $\mathbf{X} = (X,\delta_\mathbf{X})$ consists of a set $X$ and a (possibly partial) surjection $\delta_\mathbf{X}\pfunction{\Baire}{X}$.
Many spaces of interest can be represented in standard ways, such as:
$\Baire$,
$\mathbb{N}$,
$\mathbb{N}^{<\mathbb{N}}$,
initial segments of $\mathbb{N}$,
the set of binary relations on $\mathbb{N}$,
the set of $\boldfaceGamma$-definable subsets of $\mathbb{N}$ where $\boldfaceGamma$ is a pointclass in the arithmetic or projective hierarchy,
countable Cartesian products and countable disjoint unions of represented spaces. For the formal definitions of the representation maps of these spaces, we refer the reader to \cite{bgp21}, or to our previous paper \cite{goh-pauly-valenti}.

A \emph{problem} $f$ is a (possibly partial) multivalued function between represented spaces $\mathbf{X}$ and $\mathbf{Y}$, denoted $f: \subseteq \mathbf{X} \rightrightarrows \mathbf{Y}$.
For each $x \in X$, $f(x)$ denotes the set of possible outputs (i.e., \emph{$f$-solutions}) corresponding to the input $x$.
The \emph{domain} $\dom(f)$ is the set of all $x \in X$ such that $f(x)$ is non-empty.
Such an $x$ is called an \emph{$f$-instance}.
If $f(x)$ is a singleton for all $x \in \dom(f)$,
we say $f$ is \emph{single-valued} and write $f: \subseteq \mathbf{X} \to \mathbf{Y}$.
In this case, if $y$ is the $f$-solution to $x$,
we write $f(x) = y$ instead of (the formally correct) $f(x) = \{y\}$. A function $F: \subseteq \Baire \to \Baire$ is a \emph{realizer} for a problem $f$ if whenever $p$ is a name for some $x \in \dom(f)$, then $F(p)$ is a name for an $f$-solution to $x$. We say a problem is \emph{computable} (resp.\ \emph{continuous}) if it has a computable (resp.\ continuous) realizer.

A problem $f$ is \emph{Weihrauch reducible} to a problem $g$,
written $f \leqW g$,
if there are computable maps $\Phi, \Psi: \subseteq \Baire \to \Baire$ such that if $p$ is a name for some $x \in \dom(f)$, then
\begin{enumerate}
	\item $\Phi(p)$ is a name for some $y \in \dom(g)$, and
	\item if $q$ is a name for some $g$-solution of $y$,
	then $\Psi(p,q)$ is a name for some $f$-solution of $x$.
\end{enumerate}
If $\Phi$ and $\Psi$ satisfy the above, we say that $f \leqW g$ \emph{via $\Phi$, $\Psi$}. We sometimes refer to $\Phi$ and $\Psi$ as the \emph{forward} and \emph{backward} functionals respectively. We say that $f$ is \emph{strongly Weihrauch reducible} to $g$, written $f \leqsW g$, if the above holds but with $\Psi(p,q)$ replaced by $\Psi(q)$ in (2). For us, strong Weihrauch reducibility is only of occasional technical interest.

Weihrauch reducibility forms a preorder on problems.
We say $f$ and $g$ are \emph{Weihrauch equivalent}, written $f \equivW g$,
if $f \leqW g$ and $g \leqW f$.
The $\equivW$-equivalence classes (\emph{Weihrauch degrees}) are partially ordered by $\leqW$.
Among the numerous algebraic operations in the Weihrauch degrees, we consider:
\begin{itemize}
	\item for problems $f_i: \subseteq \mathbf{X}_i \rightrightarrows \mathbf{Y}_i$, the \emph{parallel product}
	\[ f_0 \times f_1: \subseteq \mathbf{X}_0 \times \mathbf{X}_1 \rightrightarrows \mathbf{Y}_0 \times \mathbf{Y}_1 \text{ defined by } (x_0,x_1) \mapsto f_0(x_0) \times f_1(x_1), \]
	i.e., given an $f_0$-instance and an $f_1$-instance, solve both;
	\item for a problem $f: \subseteq \mathbf{X} \rightrightarrows \mathbf{Y}$, the \emph{(infinite) parallelization}
	\[ \parallelization{f}: \subseteq \mathbf{X}^\mathbb{N} \rightrightarrows \mathbf{Y}^\mathbb{N} \text{ defined by } (x_i)_i \mapsto \prod_i\nolimits f(x_i), \]
	i.e., given a countable sequence of $f$-instances, solve all of them.
\end{itemize}
These operations are defined on problems,
but they all lift to the Weihrauch degrees.
Parallelization even forms a closure operator, i.e.,
$f \leqW \parallelization{f}$,
$f \leqW g$ implies $\parallelization{f} \leqW \parallelization{g}$,
and $\parallelization{\parallelization{f}} \equivW \parallelization{f}$. We also briefly mention two further operations on Weihrauch degrees, respectively called \emph{compositional product} and \emph{implication}, characterized by $f \compproduct g = \max_{\leqW} \{F \circ G \st F \leqW f \wedge G \leqW g\}$ and $f \rightarrow g = \min_{\leqW} \{h \st g \leqW f \compproduct h\}$. Both are total operations. The compositional product can be equivalently described as follows: let $\sequence{\Phi_w}{w\in\Baire}$ be an effective enumeration of all partial continuous functionals $\Baire\rightarrow\Baire$ with $G_\delta$ domain. For every $f\pmfunction{\mathbf{X}}{\mathbf{Y}}$, $g\pmfunction{\mathbf{U}}{\mathbf{V}}$, the domain of $f\compproduct g$ is the set
\[ \{ (w,u) \in \Baire \times \dom(g) \st (\forall v \in g(u))(\forall p_v\in \delta_{\mathbf{V}}^{-1}(v))(\delta_{\mathbf{X}}\Phi_w(p_v)\in \dom(f))\}.\]
Given as input a pair $(w,u)$, $f\compproduct g$ produces a pair $(y,v)$ with $v\in g(u)$ and $y\in f\delta_{\mathbf{X}}\Phi_w(p_v) $ for some $p_v\in\delta_{\mathbf{V}}^{-1}(v)$. If $f,g$ are problems on the Baire space, $f\compproduct g$ is simply the problem that maps $(w,u)$ to a pair $(y,v)$ with $v\in g(u)$ and $y\in f(v)$. For details, we refer to \cite{paulybrattka4, Westrick20diamond}.

The Weihrauch degrees also support a number of interior operators,
which have been used to separate degrees of interest (see e.g.\ \cite[\S 3.1]{soldavalenti}).
For any problem $f$ and any represented space $\mathbf{X}$, the problem
\[ \DetPartX{X}{f} := \max_{\leqW}\nolimits \{g \leqW f \st g \text{ has codomain }\mathbf{X}\text{ and is single-valued}\} \]
exists \cite[Theorem 3.2]{goh-pauly-valenti}.
We call $\DetPartX{\Baire}{f}$ the \emph{deterministic part of $f$} and denote it by $\Det{f}$ for short. We say that a problem is \emph{deterministic} if $f\weiequiv \Det{f}$, i.e.~if $f$ is Weihrauch equivalent to a single-valued function on Baire space. This terminology is motivated by the observation that the treatment of multi-valued functions in computable analysis inherently corresponds to a non-deterministic computation point of view.

Observe that \cite[Proposition 3.6]{goh-pauly-valenti} can be generalized slightly:

\begin{proposition} \thlabel{prop:Det_below_parallel_Det_2}
$\Det{f} \leqW \parallelization{\DetPartX{2}{f}}$. ($\mathbf{2}$ is the two-point space with the discrete topology.)
\end{proposition}
\begin{proof}
Suppose $g$ is single-valued, has codomain $\Baire$, and $g \leqW f$.
Define a single-valued problem $h$ with codomain $\textbf{2}$ as follows:
Given $n,m \in \mathbb{N}$ and a $g$-instance $x$,
produce $1$ if $g(x)(n) \geq m$, otherwise produce $0$.
It is easy to see that $g \leqW \parallelization{h}$ and $h \leqW f$.
The latter implies $h \leqW \DetPartX{2}{f}$ and so $g \leqW \parallelization{h} \leqW \parallelization{\DetPartX{2}{f}}$.
\end{proof}

For any problem $f$ and $\mathbf{X} = \mathbb{N}$ or $\mathbf{k}$,
it is also known that
\[ \max_{\leqW}\nolimits \{g \leqW f \st g \text{ has codomain }\mathbf{X}\} \]
exists. For $\mathbf{X} = \mathbb{N}$ we call it the \emph{first-order part of $f$} \cite[Theorem 2.2]{DSY23}, denoted by $\fop{f}$,
while for $\mathbf{X} = \mathbf{k}$ we call it the \emph{$\mathbf{k}$-finitary part of $f$} \cite[Proposition 2.9]{paulycipriani1},
denoted by $\Fin{k}{f}$. We say that a problem is \emph{first-order} if $f\weiequiv \fop{f}$. 
It is immediate from the definitions that $\DetPartX{\mathbb{N}}{f} \leqW \fop{f}$ and $\DetPartX{k}{f} \leqW \Fin{k}{f} \leqW \fop{f}$.

Before providing examples of problems of interest, we shall specify our notation for sequences and trees. A string or sequence is a function $\sigma$ from an initial segment of $\mathbb{N}$ to $\mathbb{N}$, with length denoted by $|\sigma|$. The set of finite (resp.\ infinite) strings of natural numbers is denoted by $\baire$ (resp.\ $\Baire$). Likewise, we use $n^{<\mathbb{N}}$ and $n^\mathbb{N}$ respectively for the sets of finite or infinite strings with range in $\{0,\hdots,n-1\}$. In particular, the sets of finite and infinite binary strings are denoted $\cantor$ and $\Cantor$ respectively.

The constant sequence with value $n$ and length $i$ is denoted by $n^i$. For $i = 0$, this is the empty string $\varepsilon$. We use $n^\omega$ for the constant infinite string with value $n$. The concatenation of strings $\sigma$ and $\tau$ is usually denoted by $\sigma\concat\tau$. For the sake of readability, we omit the concatenation symbol whenever we are concatenating constant strings (e.g.\ we use $0^i1^j$ in place of $0^i\concat 1^j$). The prefix relation on strings is denoted by $\sqsubseteq$.  

For each $p \in \Baire$ and $i \in \mathbb{N}$, $p[i]$ denotes the string obtained by restricting $p$ to $i$. A subtree of $\baire$ ($\cantor$ resp.) is a subset $T \subseteq \baire$ ($\cantor$ resp.) which is closed downwards under prefix. We picture trees growing upwards, i.e., the root $\varepsilon$ lies at the bottom. 

We now list some well-studied computational problems that will be useful.

\begin{itemize}
    \item $\lpo$: Given $p \in \Cantor$, produce $1$ if there is $k\in\mathbb{N}$ such that $p(k) = 1$, otherwise produce $0$.
    \item $\Choice{k}$: Given $p\in (k+1)^\mathbb{N}$ such that $(\exists n<k)(n+1\notin\ran(p))$, find $n<k$ such that $n+1\notin\ran(p)$.
    \item $\CNatural$: Given $p\in\Baire$ such that $(\exists n)(n+1\notin\ran(p))$, find~$n$ such that $n+1\notin\ran(p)$.
    \item $\mflim$: Given a convergent sequence $\sequence{p_n}{n\in\mathbb{N}}$ in $\Baire$, find its limit.
    \item $\KL$: Given an infinite finitely branching subtree of $\mathbb{N}^{<\mathbb{N}}$, find an infinite path through it.
\end{itemize}

The problems $\Choice{k}$ and $\CNatural$ are examples of \emph{choice problems}, as they can be rephrased as ``given a non-empty co-c.e.\ closed subset of $\mathbf{k}$ (resp.\ $\mathbb{N}$), find a point in it''. Choice problems are pivotal in the study of Weihrauch reducibility, as they provide a useful benchmark to describe the Weihrauch degrees of other problems. For this reason, many variations of choice problems have been introduced in the literature. For example, we denote with $\UCNatural$ the restriction of $\CNatural$ to co-c.e.\ closed singletons (or, equivalently, to sequences $p$ such that there is exactly one $n$ such that $n+1\notin\ran(p)$). It is known that $\CNatural\weiequiv \UCNatural$ (\cite[Theorem 3.8]{BolWei11}), hence $\CNatural$ is a deterministic, first-order problem. Another variant that plays an important role in this paper is \emph{all or co-unique choice}: if $\mathbf{X}$ is $\mathbf{k}$ or $\mathbb{N}$, $\ACC_\mathbf{X}$ is the problem ``Given an enumeration of a set $A \subseteq \mathbf{X}$ of size at most $1$, find a number not in $A$''. It is known that $\parallelization{\ACC_\mathbb{N}} \leW \mflim \leW \KL$ (see e.g.\ \cite{bgp21}).  

Another important family of problems comes from Ramsey's theorem for $n$-tuples and $k$-colors. In this paper, we only deal with colorings of the natural numbers. We define $\RT{1}{k}$, a.k.a.\ the pigeonhole principle, as the total multi-valued function that maps a coloring $c\function{\mathbb{N}}{k}$ to the set of all infinite $c$-homogeneous sets, i.e.\ the set of all infinite $H\subseteq \mathbb{N}$ such that $c$ is constant on $H$.  Many well-known facts about the Weihrauch degrees of Ramsey principles can be found in \cite{BRramsey17}. We only mention that (in terms of Weihrauch reducibility) we can equivalently think of $\RT{1}{k}$ as the problem that, given $c$, produces \emph{the color} of a $c$-homogeneous solution. It is known that $j<k$ implies $\RT{1}{j}\leW \RT{1}{k}$ and that $\parallelization{\RT{1}{2}}\weiequiv \KL$. In particular, $\KL$ is parallelizable (as the parallelization is a closure operator).

To study the problems $\DS$ and $\BS$ from the point of view of Weihrauch reducibility, we need to introduce the represented spaces of linear orders and quasi-orders. We only work with countable linear orders/quasi-orders with domain contained in $\mathbb{N}$. We represent a linear order $(L,\le_L)$ with the characteristic function of the set $\{ \pairing{n,m} \st n\le_L m\}$. Likewise, we represent a quasi-order $(Q,\preceq_Q)$ with the characteristic function of the set $\{ \pairing{n,m} \st n\preceq_Q m\}$. 

We conclude this section by observing a fact about $\BS$ which was implicit in \cite{goh-pauly-valenti}.

\begin{proposition} \thlabel{prop:BS_qo_po}
$\BS$ is Weihrauch equivalent to its restriction to partial orders.
\end{proposition}
\begin{proof}
Given a non-well quasi-order $(Q,\preceq_Q)$ where $Q \subseteq \mathbb{N}$,
compute the set $S := \{a \in Q \st (\forall b <_\mathbb{N} a)(a \not\preceq_Q b \text{ or }b \not\preceq_Q a)\}$.
The restriction $(S,\preceq_Q)$ is a non-well partial order because it is isomorphic to the partial order of $\preceq_Q$-equivalence classes.
\end{proof}

Henceforth we will use \thref{prop:BS_qo_po} without mention.

\section{Separating $\BS$ and $\DS$}
\label{sec:separation}
We shall separate $\BS$ and $\DS$ by separating their first-order parts.

\begin{theorem}
\thlabel{theo:bsdssep}
$\fop{\BS} \nleqW \fop{\DS}$ and so $\DS \leW \BS$.
\end{theorem}

Recall from \cite[Theorem 4.10]{goh-pauly-valenti} that $\fop{\DS} \equivW \PiBound$,
which is the problem of producing an upper bound for a finite subset of $\mathbb{N}$ (given via a $\boldfacePi^1_1$-code).
Observe that $\PiBound$ is \emph{upwards closed}, i.e., if $n \in g(x)$ then $m \in g(x)$ for all $m > n$.

\begin{lemma}
\thlabel{lemma:upwardsclosed}
Let $f$ be a problem with codomain $\mathbb{N}$. The following are equivalent:
\begin{enumerate}
\item there exists an upwards closed problem $g$ with codomain $\mathbb{N}$ such that $f\leqW g$;
\item there is a computable procedure which takes as input any $x \in \dom(f)$ and produces a sequence $p_x\in\Baire $ of guesses for $f$-solutions to $x$ which is correct cofinitely often.
\end{enumerate}
\end{lemma}
\begin{proof}
For $(1) \Rightarrow (2)$, let $g$ be upwards closed with codomain $\mathbb{N}$ and assume $f \leqW g$ via $\Phi$ and $\Psi$. Given $x \in \dom(f)$, run the computations $(\Psi(x,m))_{m \in \mathbb{N}}$ in parallel. Once some $\Psi(x,m)$ halts, we output its result and cancel $\Psi(x,n)$ for all $n < m$. This produces a sequence of numbers. The fact that $g$ is upwards closed guarantees that cofinitely many elements of this sequence are elements of $f(x)$.

For the converse direction, for every $x\in \dom(f)$, let $p_x\in\Baire$ be as in the hypothesis. Define  $M_x:=\max \{m \st p_x(m)\notin f(x)\}$ and let $g(x):=\{n \st n>M_x\}$. Clearly $g$ is upwards closed. The fact that $f\leqW g$ follows from the fact that $x\mapsto p_x$ is computable.
\end{proof}

Given a non-well quasi-order $(Q,\preceq_Q)$, we say that a finite sequence $\sigma$ is \emph{extendible} to an infinite $\preceq_Q$-bad sequence (or, more compactly, $\sigma$ is $\preceq_Q$-extendible) if there is a $\preceq_Q$-bad sequence $\sequence{q_n}{n\in\mathbb{N}}$ such that $(\forall i<\length{\sigma})(\sigma(i)=q_i)$. We omit the subscript whenever there is no ambiguity. 

Observe that $\fop{\BS}$ can compute the problem ``given a non-well partial order $(P,\le_P)$, produce an element of $P$ that is extendible to an infinite bad sequence''. In light of \thref{lemma:upwardsclosed}, to prove \thref{theo:bsdssep} it suffices to show that one cannot computably ``guess'' solutions for $\BS$. In other words, given a computable procedure which tries to guess extendible elements in a non-wqo, we want to construct a non-wqo $P$ on which the procedure outputs a non-extendible element infinitely often. This would imply that $\fop{\BS} \not\weireducible \PiBound$. The non-wqos $P$ we construct will be ``tree-like'' in the following sense:

\begin{definition} \thlabel{def:treedecomposition}
A \emph{tree decomposition} of a partial order $(P,\le_P)$ consists of a tree $T \subseteq \cantor$ and a function $\iota\function{T}{P}$ such that:
\begin{enumerate}
	\item If $w_1,w_2 \in T$ and $w_1$ is a proper prefix of $w_2$ (written $w_1 \pprefix w_2$), then $\iota(w_1) <_P \iota(w_2)$.
	\item $P$ is partitioned into finite $P$-intervals, where each interval has the form
	\[ (w\concat b] = \{v \in P \st \iota(w) <_P v \leq_P \iota(w\concat b)\} \]
	for some vertex $w\concat b \in T$ (with final entry $b$), or $(\varepsilon] = \{\iota(\varepsilon)\}$ (where $\varepsilon$ denotes the root of $\cantor$). For $v \in P$ let $\lceil v \rceil \in T$ be uniquely defined by $v \in (\lceil v \rceil]$.
	\item If $w_1, w_2 \in T$ are incompatible, so are $\iota(w_1)$ and $\iota(w_2)$ (i.e.\ they have no common upper bound in $P$).
\end{enumerate}
\end{definition}

The following lemma is straightforward.

\begin{lemma}
\thlabel{lemma:treedecomposition}
If $\iota\function{T}{P}$ is a tree decomposition, then $P$ has no infinite descending sequences. Moreover, $T$ is wqo (i.e.\ it has finite width) if and only if $P$ is wqo. In other words, $T$ has an infinite antichain iff so does $P$.
\end{lemma}
\begin{proof}
The fact that every partial order that admits a tree decomposition does not have an infinite descending sequence follows from the fact that if ($v_n)_{n \in \mathbb{N}}$ is an infinite descending sequence in $P$, then since every interval $(\lceil v_n \rceil ]$ is finite, up to removing duplicates, the sequence $(\lceil v_n\rceil)_{n\in\mathbb{N}}$ would be an infinite descending sequence in $T$. 

If $(w_n)_{n \in \mathbb{N}}$ is an infinite antichain in $T$ then, by definition of tree decomposition, $(\iota(w_n))_{n \in \mathbb{N}}$ is an infinite antichain in $P$. Conversely, if $(v_n)_{n \in \mathbb{N}}$ is an infinite antichain in $P$, then for every $n$, for all but finitely many $m$, $\lceil v_n \rceil$ is $\prefix$-incomparable with $\lceil v_m \rceil$. In particular, we can obtain an infinite antichain in $T$ by choosing a subsequence $(v_{n_i})_{i \in \mathbb{N}}$ such that, for every $i\neq j$, $\lceil v_{n_i} \rceil$ and $\lceil v_{n_j} \rceil$ are $\prefix$-incomparable.
\end{proof}  

\begin{lemma}
\thlabel{lemma:seperationwork}
There is no computable procedure that, given as input a non-well partial order which admits a tree decomposition, outputs an infinite sequence of elements of that partial order such that cofinitely many elements in the output are extendible to a bad sequence.
\end{lemma}

We point out a subtle yet important aspect regarding \thref{lemma:seperationwork}: The procedure only has access to the partial order, not to a tree decomposition of it.

\begin{proof}
Fix a computable ``guessing'' procedure $g$ that receives as input a partial order (admitting a tree decomposition) and outputs an infinite sequence of elements in that partial order. We shall build a partial order $P$ together with a tree decomposition $\iota \function{T}{P}$ in stages such that, infinitely often, $g$ outputs an element of $P$ that does not extend to an infinite bad sequence. 

Start with $T_0 = \{\varepsilon\}$ and $P_0$ having a single element $v_\varepsilon$, with $\iota_0(\varepsilon) = v_\varepsilon$. In stage $s$, we have built a finite tree decomposition $\iota_s\function{T_s}{P_s}$ and wish to extend it to some $\iota_{s+1}\function{T_{s+1}}{P_{s+1}}$. The tree $T_{s+1}$ will always be obtained by giving each leaf in $T_s$ a single child, and then adding two children to exactly one of the new leaves. To decide which leaf gets two children, say a finite extension $Q$ of $P_s$ is \emph{suitable} for $\iota_s\function{T_s}{P_s}$ if for every $v \in Q \setminus P_s$, there is exactly one leaf $w \in T_s$ such that $\iota_s(w) <_Q v$. Pick the left-most leaf $\sigma$ of $T_s$ with the following property:
\begin{quote}
There is some suitable extension $Q$ of $P_s$ such that, when given $Q$, the guessing procedure $g$ would guess an element of $Q$ which is comparable with $\iota_s(\sigma)$.
\end{quote}
To see that such $\sigma$ must exist, consider extending $P_s$ by adding an ``infinite comb'' (i.e.\ a copy of $\{0^n1^i \st n \in \mathbb{N}, i \in \{0,1\}\}$) above the $\iota_s$-image of a single leaf in $T_s$. The resulting partial order $\overline{Q}$ is non-wqo, admits a tree decomposition (obtained by extending $T_s$ and $\iota_s$ in the obvious way), and its finite approximations (extending $P_s$) are suitable for $\iota_s$. Hence, by hypothesis, $g$ eventually guesses some element, which must be comparable with $\iota_s(\sigma)$ for some leaf $\sigma \in T_s$ (because all elements of $\overline{Q}$ are).

Having identified $\sigma$, we fix any corresponding suitable extension $Q$ of $P_s$. In order to extend $\iota_s$, we further extend $Q$ to $Q'$ by adding a new maximal element $v_w$ to $Q$ for each leaf $w \in T_s$ as follows: $v_w$ lies above all $v \in Q\setminus P_s$ such that $\iota_s(w) <_P v$, and is incomparable with all other elements (including the other new maximal elements $v_{w'}$). To extend $T_s$, we add a new leaf $\tau \concat 0$ to $T_s$ for each leaf $\tau$, obtaining a tree $T'$. We extend $\iota_s$ to yield a tree decomposition $\iota'\function{T'}{Q'}$ in the obvious way.

Finally, we add two children to $\sigma\concat 0$ in $T'$, i.e., define $T_{s+1} := T' \cup \{\sigma \concat 00, \sigma\concat  01\}$. We also add two children $v_1$, $v_2$ to $\iota'(\sigma\concat 0)$ in $Q'$ to obtain $P_{s+1}$, and extend $\iota'$ to $\iota_{s+1}$ by setting $\iota_{s+1}(\sigma\concat 0i) := v_i$. This concludes stage $s$.

It is clear from the construction that $\iota: T \to P$ is a tree decomposition. Let us discuss the shape of the tree $T$. In stage $s$, we introduced a bifurcation above a leaf $\sigma_s$ of $T_s$. These are the only bifurcations in $T$. Observe that, whenever $s' < s$, $\sigma_s$ is either above or to the right of $\sigma_{s'}$, because every suitable extension of $P_s$ is also a suitable extension of $P_{s'}$ and, at stage $s'$, the chosen leaf was the left-most. This implies that the sequence $(\sigma_s)_s$ converges to a path $p\in [T]$, i.e.\ for every $n$ there is a stage $s$ such that for every $s'>s$, $p[n]\prefix \sigma_{s'}$. This also implies that $p$ is the unique non-isolated path (if $\tau \not \prefix p$ then there are only finitely many bifurcations above $\tau$). Observe also that a vertex $w$ in $T$ is extendible to an infinite antichain in $T$ if and only if it does not belong to $p$.

\begin{figure}[htb]
	\centering
	\begin{tikzpicture}[scale=0.7]
		\begin{pgfonlayer}{nodelayer}
			\node [style=pnode] (0) at (0, 0) {};
			\node [style=pnode] (1) at (-3, 2) {};
			\node [style=pnode] (2) at (3, 2) {};
			\node [style=pnode] (3) at (-4.5, 4) {};
			\node [style=pnode] (4) at (-1.5, 4) {};

			\node (4l) at (-0.5, 4) {$\iota(\sigma)$};
			\node [style=extnode] (5) at (-5.5, 6) {};
			\node [style=extnode] (6) at (-4.5, 6) {};
			\node [style=extnode] (7) at (-3.5, 6) {};
			\node [style=extnode] (8) at (-4, 7) {};
			\node [style=extnode] (9) at (-1.5, 5.5) {};
			\node [fill=gray, draw=black, inner sep=3pt, shape=rectangle] (10) at (-1.5, 6.25) {};
%						\node [color=red] (10l) at (0, 6.25) {$g(s+1)$};
			\node [style=extnode] (11) at (-1.5, 7) {};

			%10
			\node [style=extnode] (13) at (2.5, 5.5) {};
			%11
			\node [style=extnode] (14) at (3.5, 5.5) {};
			%100
			\node [style=extnode] (16) at (2, 6.5) {};
			%101
			\node [style=extnode] (15) at (3, 6.5) {};
			%111
			\node [style=extnode] (17) at (4, 6.5) {};

			\node [style=extnode] (18) at (-5, 7) {};

			\node [style=pnode] (19) at (-4.5, 8) {};
				\node [style=pnode] (20) at (3, 7.5) {};
				\node [style=pnode] (11b) at (-1.5, 7) {};

				\node (4l) at (-0.1, 7) {$\iota(\sigma\concat 0)$};
				\node (g) at (0, 6.3) {$g(s+1)$};
			\node [style=pnode] (21) at (-3, 9.5) {};
			\node [style=pnode] (22) at (0, 9.5) {};
		\end{pgfonlayer}
		\begin{pgfonlayer}{edgelayer}
				\draw (1) to (0);
				\draw (0) to (2);
				\draw (1) to (3);
				\draw (1) to (4);

				\draw (13) to (15);
				\draw (14) to (17);
				\draw (15) to (14);
				\draw (13) to (16);
				\draw (4) to (9);
				\draw (3) to (5);
				\draw (3) to (6);
				\draw (3) to (7);
				\draw (5) to (18);
				\draw (6) to (8);
				\draw (18) to (6);
				\draw (8) to (7);
				\draw (9) to (10);
				\draw (10) to (11);
				\draw (2) to (13);
				\draw (2) to (14);

				\draw (8) to (19);
				\draw (18) to (19);
				\draw (16) to (20);
				\draw (15) to (20);
				\draw (17) to (20);

				\draw (11) to (21);
				\draw (11) to (22);

			\end{pgfonlayer}
		\begin{pgfonlayer}{background}

			\node[ngroup, fit=(0)] {};
			\node[ngroup, fit=(1)] {};
			\node[ngroup, fit=(2)] {};
			\node[ngroup, fit=(3)] {};
			\node[ngroup, fit=(4)] {};
			
			\node[ngroup, fit=(5) (6) (7) (8) (18) (19)] {};
			\node[ngroup, fit=(9) (10) (11)] {};
			\node[ngroup, fit=(13) (14) (15) (16) (17) (20)] {};

			\node[ngroup, fit=(21)] {};
			\node[ngroup, fit=(22)] {};

			\node[draw=gray, dashed, rounded corners, inner sep=1em, fit=(0) (1) (2) (3) (4)] {};

		\end{pgfonlayer}
	\end{tikzpicture}
	\caption{Schematic representation of the construction used in the proof of \thref{lemma:seperationwork}. The dashed box contains the partial construction up to stage $s$. The gray boxes contain intervals in the partial order $P$. For simplicity, each interval up to stage $s$ only contains one point (i.e.\ the partial order $P_s$ is isomorphic to a tree), but this need not be the case in general. The black nodes are those in the range of the (partial) tree decomposition $\iota_s$. The square gray node is the node guessed by $g$ at stage $s+1$ (which identifies $\iota(\sigma)$ and $\iota(\sigma \concat 0)$). }
\end{figure}

We may now apply \thref{lemma:treedecomposition} to analyze $P$. First, since $T$ is not wqo, neither is $P$. Second, we claim that if $v <_P \iota(\sigma)$ for some vertex $\sigma$ on $p$, then $v$ is not extendible to an infinite bad sequence. To prove this, suppose $v$ is extendible. Then so is $\iota(\sigma)$. The proof of \thref{lemma:treedecomposition} implies that $\lceil \iota(\sigma) \rceil = \sigma$ is extendible to an infinite antichain in $T$. So $\sigma$ cannot lie on $p$, proving our claim.

To complete the proof, observe that our construction of $\iota$ ensures that for each $s$, $g(P)$ eventually outputs a guess which is below $\iota(\sigma_s \concat 0)$. Whenever $\sigma_s \concat 0$ lies along $p$ (which holds for infinitely many $s$), this guess is wrong by the above claim.
\end{proof}

We may now complete the proof of \thref{theo:bsdssep}.

\begin{proof}[Proof of \thref{theo:bsdssep}]
Suppose towards a contradiction that $\fop{\BS} \leqW \PiBound$. Since the problem of finding an element in a non-wqo which extends to an infinite bad sequence is first-order, it is Weihrauch reducible to $\PiBound$ as well. Now, $\PiBound$ is upwards closed, so there is a computable guessing procedure for this problem (\thref{lemma:upwardsclosed}). However such a procedure cannot exist, even for partial orders which admit a tree decomposition (\thref{lemma:seperationwork}).
\end{proof}

\section{Separating $\KL$ and $\BS$}
\label{sec:klbs}

In this section, we answer two of the main questions that were left open in \cite[Question 6.1 and 6.2]{goh-pauly-valenti}, namely whether the problems $\KL$ and $\mathsf{wList}_{2^{\mathbb{N}},\leq\omega}$ are Weihrauch reducible to $\DS$. We already introduced $\KL$ in Section~\ref{sec:background}, while $\mathsf{wList}_{2^{\mathbb{N}},\leq\omega}$ is the problem of enumerating all elements (possibly with repetition) of a given non-empty countable closed subset of $2^\mathbb{N}$. This problem was introduced in \cite[\S 6]{kmp20} and studied also in \cite{CMVCantorBendixson}. 

In fact, we prove something stronger, namely that neither $\KL$ nor $\mathsf{wList}_{2^{\mathbb{N}},\leq\omega}$ are Weihrauch reducible to $\BS$. The core of the proofs rests on the following two technical results.

\begin{definition} \thlabel{def:bad_seq_quasiorder}
Given a fixed partial order $(P,\leq_P)$, we define the following quasi-order on the (finite or infinite) $\leq_P$-bad sequences:
\[ \alpha \trianglelefteq^P \beta \defiff \alpha = \beta \text{ or } (\exists i<\length{\alpha})(\forall j<\length{\beta})(\alpha(i) \leq_P \beta(j)). \]
We just write $\trianglelefteq$ when the partial order is clear from the context. 
\end{definition}

\begin{lemma}
	\thlabel{lem}
	Let $(P,\le_P)$ be a non-well partial order and let $\alpha,\beta$ be finite $\le_P$-bad sequences. If $\alpha\trianglelefteq \beta$ and $\alpha$ is extendible to an infinite $\le_P$-bad sequence, then so is $\beta$. If $\alpha$ is not extendible then there is an infinite $\le_P$-bad sequence $B\in \Baire$ such that $\alpha\trianglelefteq B$. (Hence $\alpha \trianglelefteq \beta$ for every initial segment $\beta$ of $B$.)
\end{lemma}
\begin{proof}
	To prove the first part of the theorem, fix $\alpha\trianglelefteq \beta$ and let $A\in\Baire$ be an infinite $\le_P$-bad sequence extending $\alpha$. Let also $i<\length{\alpha}$ be a witness for $\alpha\trianglelefteq \beta$. For every $j>i$ and every $k<\length{\beta}$, $\beta(k)\not\le_P A(j)$ (otherwise $A(i) = \alpha(i) \le_P\beta(k)\le_P A(j)$ would contradict the fact that $A$ is a $\le_P$-bad sequence), which implies that $\beta$ is extendible.

	Assume now that $\alpha$ is non-extendible and let $F \in \Baire$ be a $\le_P$-bad sequence. We show that there is $i<\length{\alpha}$ and infinitely many $k$ such that $\alpha(i) <_P F(k)$. This is enough to conclude the proof, as we could take $B$ as any subsequence of $F$ with $\alpha(i) <_P B(k)$ for every $k$ (i.e.\ $\alpha \trianglelefteq B$). 

	Assume that, for every $i<\length{\alpha}$ there is $k_i$ such that for every $k\ge k_i$, $\alpha(i)\not \le_P F(k)$ (since $P$ is a partial order, there can be at most one $k$ such that $\alpha(i)=F(k)$). Since $\alpha$ is finite, we can take $k:=\max_{i<\length{\alpha}} k_i$ and consider the sequence which extends $\alpha$ by the tail $F(k+1), F(k+2), \dots$ of $F$. We have now reached a contradiction as this is an infinite $\le_P$-bad sequence extending $\alpha$. 
\end{proof}

\begin{theorem}
	\thlabel{theo:dsproducts}
	Let $f$ be a problem. The following are equivalent:
	\begin{enumerate}
		\item $f \leqW \mflim$
		\item $\parallelization{\ACC_\mathbb{N}} \times f \leqW \mflim$
		\item $\parallelization{\ACC_\mathbb{N}} \times f \leqW \DS$
		\item $\parallelization{\ACC_\mathbb{N}} \times f \leqW \BS$
	\end{enumerate}
\end{theorem}

This theorem provides another example of a ``parallel quotient'' (cf.\ \cite[Remark 3.11]{dghpp20}). See Section \ref{sec:quotient} for details.

\begin{proof}
	The implication from (1) to (2) holds because $\parallelization{\ACC_\mathbb{N}} \times \mflim \equivW \mflim$. The implications from (2) to (3) and from (3) to (4) hold because $\mflim \leqW \DS \leqW \BS$ \cite[Theorem 4.16]{goh-pauly-valenti}.
	For the implication from (4) to (1), we consider a name $x$ for an input to $f$ together with witnesses $\Phi,\Psi$ for the reduction to $\BS$.
	We show that, from them, we can uniformly compute an input $q$ to $\parallelization{\ACC_\mathbb{N}}$ together with an enumeration of the set $W$ of all finite sequences $\sigma$ in the (non-well) partial order $P$ built by $\Phi$ on $(q,x)$ such that $\sigma$ does not extend to an infinite bad sequence. We can then use $\mflim$ to obtain the characteristic function of $W$.
	Having access to this lets us find an infinite bad sequence in $P$ greedily by avoiding sequences in $W$.
	From such a bad sequence (and $(q,x)$), $\Psi$ then computes a solution to $f$ for $x$.
	
	It remains to construct $q = \sequence{q_\sigma}{\sigma\in\mathbb{N}}$ and $W$ to achieve the above.
	At the beginning, $W$ is empty, and we extend each $q_\sigma$ in a way that removes no solution from its $\ACC_\mathbb{N}$-instance.
	As we do so, for each $\sigma \notin W$ (in parallel), we monitor whether the following condition has occurred:
	\begin{quote}
	$\sigma$ is a bad sequence in $P$ (as computed by the finite prefix of $(q,x)$ built/observed thus far),
	and there is some (finite) bad sequence $\tau$ in $P$ such that
	\begin{itemize}
		\item $\sigma \trianglelefteq^P \tau$,
		\item the functional $\Psi$, upon reading the current prefix of $(q,x)$ and $\tau$,
		produces some output $m$ for the $\ACC_\mathbb{N}$-instance indexed by $\sigma$.
	\end{itemize} 
	\end{quote}
	Once the above occurs for $\sigma$ (if ever), we remove $m$ as a valid solution to $q_\sigma$ by enumerating it.
	This ensures that $\tau$, and hence $\sigma$ by \thref{lem}, cannot extend to an infinite bad sequence in $P$.
	We shall then enumerate $\sigma$ into $W$.
	This completes our action for $\sigma$, after which we return to monitoring the above condition for sequences not in $W$.
	
	It is clear that each $q_\sigma$ is an $\ACC_\mathbb{N}$-instance (with solution set $\mathbb{N}$ if the condition is never triggered, otherwise with solution set $\mathbb{N}\setminus\{m\}$).
	Hence $P:=\Phi(q,x)$ is a non-well partial order.
	As argued above, no $\sigma \in W$ extends to an infinite bad sequence in $P$.
	Conversely, suppose $\sigma$ is a bad sequence in $P$ which does not extend to an infinite bad sequence.
	Since $P$ is non-wqo, by \thref{lem} there is an infinite bad sequence $r$ such that $\sigma \trianglelefteq^P r$.
	Then $\Psi$ has to produce all $\ACC_\mathbb{N}$ answers upon receiving $(q,x)$ and $r$,
	including an answer to $q_\sigma$. By continuity, this answer is determined by finite prefixes only. In particular, after having constructed a sufficiently long prefix of $q$,
	some finite prefix $\tau$ of $r$ will trigger the condition for $\sigma$ (unless something else triggered it previously), which ensures that $\sigma$ gets placed into $W$.
	This shows that $W$ contains exactly the non-extendible finite bad sequences in $P$,
	thereby concluding the proof.
\end{proof}
	
In other words, this theorem can be restated saying that $\mflim$ is the strongest problem such that its parallel product with $\parallelization{\ACC_\mathbb{N}}$ is reducible to either $\mflim$, $\DS$, $\BS$. In symbols,
\begin{align*}
	\mflim & \weiequiv \max_{\weireducible} \{ h \st \parallelization{\ACC_\mathbb{N}} \times h \weireducible \mflim\}\\
	 	& \weiequiv\max_{\weireducible} \{ h \st \parallelization{\ACC_\mathbb{N}} \times h \weireducible \DS\} \\
		& \weiequiv \max_{\weireducible} \{ h \st \parallelization{\ACC_\mathbb{N}} \times h \weireducible \BS\}. 
\end{align*}
This characterizes the ``parallel quotient'' of $\DS$ and $\BS$ over $\parallelization{\ACC_\mathbb{N}}$. The parallel quotient will be discussed more extensively in Section \ref{sec:quotient}.

\begin{corollary} \thlabel{cor:ACC_DS_lim}
If $f$ is a parallelizable problem (i.e., $f \equivW \parallelization{f}$) with $\ACC_\mathbb{N} \leqW f \leqW \BS$, then $f \leqW \mflim$.
\end{corollary}
\begin{proof}
Since $\ACC_\mathbb{N} \leqW f \leqW \BS$ and $f$ is parallelizable, we have $\parallelization{\ACC_\mathbb{N}} \times f \leqW f \leqW \BS$. By the previous theorem, $f \leqW \mflim$.
\end{proof}

The negative answer to \cite[Question 6.1]{goh-pauly-valenti} is an immediate consequence of the following result:

\begin{corollary} \thlabel{cor:KL_DS}
$\KL \nleqW \BS$.
\end{corollary}
\begin{proof}
	Recall that $\KL$ is parallelizable, $\ACC_\mathbb{N} \leqW \KL$, yet $\KL \nleqW \mflim$, hence the claim follows from \thref{cor:ACC_DS_lim}.
\end{proof}

Similarly, the negative answer to \cite[Question 6.2]{goh-pauly-valenti} can be obtained as follows:

\begin{corollary} \label{cor:wList_DS}
$\mathsf{wList}_{2^{\mathbb{N}},\leq\omega} \nleqW \BS$.
\end{corollary}
\begin{proof}
	As proved in \cite[Propositions 6.12, 6.13]{kmp20}, $\mathsf{wList}_{2^{\mathbb{N}},\leq\omega}$ is parallelizable and $\ACC_\mathbb{N} \leqW \mflim \leqW \mathsf{wList}_{2^{\mathbb{N}},\leq\omega}$. On the other hand, $\mathsf{wList}_{2^{\mathbb{N}},\leq\omega} \nleqW \mflim$ \cite[Corollary 6.16]{kmp20}, hence the claim follows from \thref{cor:ACC_DS_lim}.
\end{proof}

Continuous Weihrauch reducibility ($\leqW^*$) is a variant of Weihrauch reducibility defined via continuous functions in place of computable functions. Using a recent result of Pauly and Sold\`a \cite{paulysolda},
we can characterize the parallelizations of first-order problems which are reducible to $\DS$ up to continuous Weihrauch reducibility.

\begin{corollary}
If $\parallelization{f} \leqW^* \DS$, then $\fop{f} \leqW^* \CNatural$.
Therefore, for any first-order $f$,
\[ \parallelization{f} \leqW^* \DS \qquad \text{if and only if} \qquad f \leqW^* \CNatural. \]
\end{corollary}
\begin{proof}
If $\fop{f}$ is continuous, the conclusion of the first statement is satisfied.
Otherwise, $\ACC_\mathbb{N} \leqW^* f$ by \cite[Theorem 1]{paulysolda}.
The relativization of \thref{theo:dsproducts} then implies $\parallelization{f} \leqW^* \mflim$.
We conclude $\fop{f} \leqW^* \fop{\mflim} \equivW \CNatural$.
The second statement then follows from $\parallelization{\CNatural} \equivW \mflim \leqW \DS$.
\end{proof}

\section{The finitary part and deterministic part of $\BS$}
\label{sec:fin_BS_DS}

In this section, we show that $\BS$ and $\DS$ cannot be separated by looking at their respective finitary or deterministic parts. Recall from \cite[Theorems 4.16, 4.31]{goh-pauly-valenti} that $\Det{\DS} \weiequiv \mflim$ and $\Fin{k}{\DS} \weiequiv \RT{1}{k}$. Since both the deterministic and the finitary parts are monotone, this implies that $\mflim \weireducible \Det{\BS}$ and $\RT{1}{k}\weireducible \Fin{k}{\BS}$, so we only need to show that the converse reductions hold.

Let $(Q,\preceq_Q)$ be a countable quasi-order. We call a subset $A \subseteq Q$ \emph{dense} if for every $w \in Q$ there is some $u \succeq_Q w$ with $u \in A$. We call it \emph{upwards-closed} if $w \in A$ and $w \preceq_Q u$ implies $u \in A$. By $\SDDCC$ we denote the following problem: Given a (non-empty) quasi-order $(Q,\preceq_Q)$ and a $\boldfaceSigma^1_1$-code for a dense upwards-closed subset $A \subseteq Q$, find some element of $A$. 
We can think of a $\boldfaceSigma^1_1$-code for $A$ as a sequence $\sequence{T_n}{n\in Q}$ of subtrees of $\baire$ such that $T_n$ is ill-founded if and only if $n \in A$.

\begin{proposition}
	\thlabel{thm:FOP(BS)<=SDDCC}
$\fop{\BS} \leqW \SDDCC$.
\end{proposition}
\begin{proof}
	Let $f$ be a problem with codomain $\mathbb{N}$ and assume $f\weireducible \BS$ via $\Phi,\Psi$. Fix $x\in \dom(f)$.
	Let $(P,\leq_P)$ denote the non-well partial order defined by $\Phi(x)$.
	We say that a finite $\leq_P$-bad sequence $\beta$ is \emph{sufficiently long} if $\Psi(x, \beta)$ returns a natural number in at most $\length{\beta}$ steps.

Consider the non-empty $\Sigma^{1,x}_1$ set of sufficiently long finite extendible bad sequences. To show that $f\weireducible \SDDCC$, it is enough to notice that \thref{lem} implies that the aforementioned set is dense and upwards-closed with respect to the quasi-order $\trianglelefteq^P$ (\thref{def:bad_seq_quasiorder}).
\end{proof}

\begin{lemma}
	\thlabel{thm:sddcc_cantor}
$\SDDCC \equivW \SDDCC(\cantor, \cdot)$, where the latter denotes the restriction of the former to $\cantor$ with the prefix ordering $\prefix$.
\end{lemma}
\begin{proof}
	Clearly, we only need to show that $\SDDCC \leqW \SDDCC(\cantor, \cdot)$. Suppose we are given as input a countable quasi-order $(Q,\preceq_Q)$ and a $\boldfaceSigma^1_1$-code for a set $A \subseteq Q$ which is dense and upwards-closed. We shall define a labelling $\lambda\function{\cantor}{Q}$ which is computable uniformly in $(Q,\preceq_Q)$, such that $\lambda^{-1}(A)$ is dense and upwards-closed. This suffices to prove the claimed reduction, as the preimage of $A$ via $\lambda$ is uniformly $\Sigma^1_1$ relative to the input and, given $\sigma\in \lambda^{-1}(A)$, we may use the input $(Q,\preceq_Q)$ to compute $\lambda(\sigma)\in A$.

    The labelling $\lambda: \cantor \to Q$ is defined via an auxiliary function $\overline{\lambda}\function{Q\times \cantor}{Q}$ (useful to describe $\lambda$ in a recursive manner). First, for every $i \in \mathbb{N}$, $\overline{\lambda}(x, 0^i):=\overline{\lambda}(x, 0^i1):=  x$. To evaluate $\overline{\lambda}$ on other strings, we fix a standard $Q$-computable enumeration $\sequence{x_n}{n}$ of $Q$ and define, recursively:
	\[ \overline{\lambda}(x, 0^i1b\concat \sigma) := \begin{cases}
		\overline{\lambda}(x_i,\sigma) & \text{if } b=1 \text{ and } x \preceq_Q x_i \\
		\overline{\lambda}(x,\sigma) & \text{if } b=0 \text{ or } x \not\preceq_Q x_i. 
	\end{cases} \]
	We then define $\lambda(\sigma):=\overline{\lambda}(x_0,\sigma)$. It is clear that $\lambda$ is total and it is uniformly computable in $(Q,\preceq_Q)$. Rephrasing the definition, we can observe that, for every $x$ and $\sigma$, $\overline{\lambda}(x, 0^i10\concat\sigma) = \overline{\lambda}(x, \sigma)$. Besides, $\overline{\lambda}(x, 0^i11\concat \sigma) = \overline{\lambda}(\overline{\lambda}(x, 0^i11), \sigma)$. In particular, 
	\[ 
	\overline{\lambda}(x, 0^i11\concat \sigma)=\begin{cases}
		\overline{\lambda}(x_i,\sigma) & \text{if } x\preceq_Q x_i\\
		\overline{\lambda}(x,\sigma) 	& \text{otherwise}.
	\end{cases}
	\]
	Intuitively, $\lambda$ guides the construction of an ascending $\preceq_Q$-sequence from $x_0$. Indeed, we will show that $\lambda$ is weakly monotone, and therefore each $\sigma$ induces a finite ascending $\preceq_Q$-sequence from $x_0$ to $\lambda(\sigma)$. At the same time, the labelling is designed to be highly redundant, which we use to prove that $\lambda^{-1}(A)$ is dense. 

    To show that $\lambda^{-1}(A)$ is upwards-closed, we claim that for every $\sigma,\tau \in \cantor$, $\lambda(\sigma)\preceq_Q \lambda(\sigma\concat \tau)$. The claim is proved by first using structural induction on $\sigma$ to show that for all $x \in Q$, we have $x \preceq_Q \overline{\lambda}(x,\sigma)$. Using this fact, one can perform another structural induction on $\sigma$ to show that for all $x$ and $\tau$, we have $\overline{\lambda}(x,\sigma) \preceq_Q \overline{\lambda}(x,\sigma\concat\tau)$. Taking $x = x_0$ yields $\lambda(\sigma) \preceq_Q \lambda(\sigma\concat\tau)$. Since $A$ is upwards-closed in $Q$, it follows that $\lambda^{-1}(A)$ is upwards-closed in the prefix ordering.

    To prove that $\lambda^{-1}(A)$ is dense, fix $\rho\in \cantor$ and assume $\lambda(\rho)\notin A$. Since $A$ is dense, there is some $n$ such that $\lambda(\rho)\preceq_Q x_n \in A$. Observe that $\overline{\lambda}(\lambda(\rho), 0^n11) = \overline{\lambda}(x_n, \varepsilon) = x_n$, so it suffices to find some $\rho' \sqsupseteq \rho$ such that $\lambda(\rho') = \overline{\lambda}(\lambda(\rho), 0^n11)$. The construction of $\rho'$ depends on the form of $\rho$ (all the following claims can be proven by structural induction on $\rho$):
    \begin{itemize}
        \item if $\rho$ ends with $0$, then $\overline{\lambda}(x,\rho \concat 10 \concat \tau) = \overline{\lambda}(\overline{\lambda}(x,\rho), \tau)$, so $\lambda(\rho \concat 10^{n+1}11) = \overline{\lambda}(\lambda(\rho),0^n11)$;
        \item if $\rho$ ends with an odd number of $1$s, then $\overline{\lambda}(x,\rho \concat 0 \concat \tau) = \overline{\lambda}(\overline{\lambda}(x, \rho), \tau)$, so $\lambda(\rho \concat 0^{n+1}11) = \overline{\lambda}(\lambda(\rho), 0^n11)$;
        \item otherwise, $\rho$ is either $\varepsilon$ or ends with a positive even number of $1$s. Then $\overline{\lambda}(x,\rho \concat \tau) = \overline{\lambda}(\overline{\lambda}(x, \rho), \tau)$, so $\lambda(\rho \concat 0^n11) = \overline{\lambda}(\lambda(\rho), 0^n11)$.
    \end{itemize}
    	In all cases, we have found an extension of $\rho$ which maps to $x_n$, as desired.
\end{proof}

\begin{theorem} \thlabel{theo:finkbs}
	$\Fin{k}{\BS} \equivW \Fin{k}{\SDDCC} \equivW \Fin{k}{\DS} \equivW \RT{1}{k}$ for every $k$.
\end{theorem}
\begin{proof}
	We have $\RT{1}{k} \equivW \Fin{k}{\DS} \leqW \Fin{k}{\BS} \leqW \Fin{k}{\SDDCC}$ by \cite[Theorem 4.31]{goh-pauly-valenti} and \thref{thm:FOP(BS)<=SDDCC}.
	It remains to show that $\Fin{k}{\SDDCC} \weireducible \RT{1}{k}$. By \thref{thm:sddcc_cantor}, it is enough to show that $\Fin{k}{\SDDCC(\cantor, \cdot)} \weireducible \RT{1}{k}$.
	
	Let $f$ be a problem with codomain $k$ and assume $f\weireducible\SDDCC(\cantor, \cdot)$ via $\Phi,\Psi$. Observe that every $x\in\dom(f)$ induces a coloring $c\function{\cantor}{k}$ as follows: run $\Psi(x,\sigma)$ in parallel on every $\sigma\in \cantor$. Whenever we see that $\Psi(x,\sigma)$ returns a number less than $k$, we define $c(\tau):=\Psi(x,\sigma)$ for every $\tau \prefix \sigma$ such that $c(\tau)$ is not defined yet.
	By density of the set coded by $\Phi(x)$, $c$ is total. 

	By the Chubb-Hirst-McNicholl tree theorem \cite{CHM09}, there is some $\rho \in 2^{<\mathbb{N}}$ and some color $i<k$ such that $i$ appears densely above $\rho$. We claim that such $i$ would be an $f$-solution to $x$.
	To prove this, fix (by density) some $\rho' \extend \rho$ which lies in the set coded by $\Phi(x)$. Then fix some $\tau \extend \rho'$ with color $i$. Finally, fix some $\sigma \extend \tau$ such that $c(\tau)$ was defined to be $\Psi(x,\sigma)$.
	Now $\sigma$ lies in the set coded by $\Phi(x)$ (by upwards-closure), so $\Psi(x,\sigma)$ is an $f$-solution to $x$. Since $\Psi(x,\sigma) = c(\tau) = i$, we conclude that $i$ is an $f$-solution to $x$. 

    This implies that $f$ reduces to the problem ``given a $k$-coloring of $\cantor$, find $i$ such that, for some $\sigma$, $i$ appears densely above $\sigma$''. We claim that the latter can be solved using $\RT{1}{k}$. This follows from the fact that, as shown in \cite[Corollary 42]{paulypradicsolda}, $\RT{1}{k}$ can solve the problem $(\mathsf{c}(\eta)^1_{<\infty})_k$ defined as ``given a $k$-coloring $\alpha\function{\mathbb{Q}}{k}$, produce $i<k$ such that, for some interval $I$ with rational endpoints, $\alpha^{-1}(i)$ is dense in $I$''.

    We introduce the order $\prec$ on $\cantor$ defined by setting $\sigma\concat0\concat\tau \prec \sigma \prec \sigma\concat1\concat\rho$ for all $\sigma, \tau, \rho \in \cantor$. This is a computable dense linear order with no endpoints, therefore there is a computable order isomorphism $\varphi$ between $(\cantor, \prec)$ and $(\mathbb{Q},<)$. 
    
    Every open interval in $(\mathbb{Q},<)$ contains the $\varphi$-image of an upper cone in $(\cantor, \prec)$ and vice versa. Specifically, the interval $(\varphi(\sigma), \varphi(\tau))$ contains the  image of the cone above $\sigma\concat 1$ if $\sigma\concat 1$ is not a prefix of $\tau$, and the image of the cone above $\tau\concat0$ if $\sigma\concat 1$ is a prefix of $\tau$. Thus, by translating the coloring along $\varphi$ to $\mathbb{Q}$, we find that a color which appears densely in an interval in $\mathbb{Q}$ also appears densely in a cone in $\cantor$. This means we can use $(\mathsf{c}(\eta)^1_{<\infty})_k$ to find a color appearing densely in a cone.
\end{proof}

It is immediate from the previous theorem that the finitary parts of $\BS$ and $\DS$ in the sense of Cipriani and Pauly \cite[Definition 2.10]{paulycipriani1} agree as well.
Finally, we shall prove that the deterministic parts of $\BS$ and $\DS$ agree.

\begin{lemma}
If $\Fin{2}{f} \leqW \RT{1}{2}$, then $\Det{f} \leqW \mflim$.
\end{lemma}
\begin{proof}
By algebraic properties of $\Det{\cdot}$, $\Fin{k}{\cdot}$ and the parallelization, we have
\[ \Det{f} \leqW \parallelization{\DetPartX{2}{f}} \leqW \parallelization{\Fin{2}{f}} \leqW \parallelization{\RT{1}{2}}. \]
Since $\Det{f}$ is deterministic, it remains to show that $\Det{\parallelization{\RT{1}{2}}} \leqW \mflim$. This is known, but for completeness we give a proof using results from the survey of Brattka, Gherardi, Pauly \cite{bgp21}:
\begin{align*}
\parallelization{\RT{1}{2}} &\equivsW \parallelization{\Choice{2}'} & \cite[11.8.11, 11.6.10]{bgp21}\phantom{.} \\
&\equivsW \left(\parallelization{\Choice{2}}\right)' & \cite[11.6.12]{bgp21}\phantom{.} \\
&\equivsW \Choice{\Cantor}' & \cite[11.7.23]{bgp21}\phantom{.} \\
&\equivW \Choice{\Cantor} \compproduct \mflim & \cite[11.7.6, 11.6.14]{bgp21}.
\end{align*}
By choice elimination (see \cite[11.7.25]{bgp21} or \cite[Theorem 3.9]{goh-pauly-valenti}), $\Det{\Choice{\Cantor} \compproduct \mflim} \leqW \mflim$. Therefore $\Det{\parallelization{\RT{1}{2}}} \leqW \mflim$.
\end{proof}

Since $\Det{\BS} \geqW \Det{\DS} \equivW \mflim$ \cite[Theorem 4.16]{goh-pauly-valenti}, we conclude that:

\begin{corollary}
	\thlabel{cor:det(bs)}
	$\Det{\BS}\weiequiv \Det{\DS}\weiequiv \mflim$.
\end{corollary}

\begin{corollary}
	$\DetPartX{\mathbb{N}}{\BS} \weiequiv \CNatural$.
\end{corollary}
\begin{proof} Since $\mathbb{N}$ computably embeds in $\Baire$, for every problem $f$ we have $\DetPartX{\mathbb{N}}{f}\weireducible \Det{f}$. In particular, by \thref{cor:det(bs)}, $\DetPartX{\mathbb{N}}{\BS} \weireducible\Det{\BS}  \weiequiv  \mflim$. Since $\fop{\mflim}\weiequiv \CNatural$ (\cite[Proposition 13.10]{BolWei11}, see also \cite[Theorem 7.2]{soldavalenti}), this implies  $\DetPartX{\mathbb{N}}{\BS}\weireducible \CNatural$. The converse reduction follows from the fact that $\CNatural\weiequiv \DetPartX{\mathbb{N}}{\DS}$ \cite[Proposition 4.14]{goh-pauly-valenti}.
\end{proof}

We remark that for establishing $\Fin{k}{\SDDCC} \weireducible \RT{1}{k}$ in \thref{theo:finkbs} it was immaterial that the set of correct solutions was provided as a $\boldfaceSigma^1_1$-set. If we consider any other represented point class $\boldsymbol{\Gamma}$ which is effectively closed under taking preimages under computable functions, and define $\boldsymbol{\Gamma}\mathsf{-DUCC}$ in the obvious way, we can obtain:

\begin{corollary}
$\Fin{k}{\boldsymbol{\Gamma}\mathsf{-DUCC}} \weireducible \RT{1}{k}$.
\end{corollary}

This observation could be useful e.g.~for exploring the Weihrauch degree of finding bad arrays in non-better-quasi-orders (cf.~\cite{freund2023logical}).

\begin{proposition} \thlabel{prop:det(fop(BS))}
    $\CNatural \weiequiv \Det{\PiBound} \weiequiv \Det{\fop{\BS}} \weiequiv \Det{\SDDCC}$.
\end{proposition}
\begin{proof}
	The reductions $\CNatural \weireducible \Det{\PiBound} \weireducible \Det{\fop{\BS}}\weireducible \Det{\SDDCC}$ are straightforward: the first one follows from $\CNatural\weiequiv \UCNatural$ and $\CNatural\weireducible \PiBound$, the second one follows from $\PiBound\weireducible \fop{\BS}$, and the last one follows from \thref{thm:FOP(BS)<=SDDCC}. It remains to show that $\Det{\SDDCC}\weireducible \CNatural$. In light of \thref{thm:sddcc_cantor}, it suffices to show that if $f\pfunction{\Baire}{\Baire}$ is such that $f\weireducible \SDDCC(\cantor, \cdot)$ via $\Phi,\Psi$ then $f\weireducible\CNatural$. Given $x\in\dom(f)$, we can compute the set 
	\begin{align*}
		X:=\{ \sigma \in \cantor \st & (\forall \tau_0,\tau_1\extend \sigma)(\forall n\in\mathbb{N}) \\
		& ((\Psi(x,\tau_0)(n)\downarrow \land \Psi(x,\tau_1)(n)\downarrow)\rightarrow \Psi(x,\tau_0)(n)=\Psi(x,\tau_1)(n)) \}.
	\end{align*}
	Observe that this is a $\Pi^{0,x}_1$ subset of $\cantor$ and it is non-empty: indeed, letting $A\subseteq\cantor$ be the set described by $\Phi(x)$, the fact that $f$ is deterministic implies that $A\subseteq X$. We can therefore use $\CNatural$ to compute some $\sigma \in X$. Observe that, by the density of $A$, there is $\tau \extend \sigma$ such that $\tau \in A$. In particular, $\Psi(x,\sigma)=\Psi(x,\tau)\in f(x)$.
\end{proof}

\thref{cor:det(bs)} and \thref{prop:det(fop(BS))} imply that $\fop{\Det{\BS}} \equivW \fop{\mflim} \equivW \CNatural \equivW \Det{\fop{\BS}}$. In general, $\fop{\Det{f}} \leqW \Det{\fop{f}}$ \cite[Proposition 3.4]{goh-pauly-valenti}. 
Goh, Pauly, Valenti \cite[Question 3.5]{goh-pauly-valenti} asked if there is some problem $f$ such that $\fop{\Det{f}} <_\mathrm{W} \Det{\fop{f}}$. The next section is devoted to an affirmative answer (\thref{thm:fop_Det_below_Det_fop}).

\section{G\"odelization}
\label{sec:godel}

In this section, we show that the first-order part operator and the deterministic part operator do not commute, answering our earlier question \cite[Question 3.5]{goh-pauly-valenti}. For this, we introduce a dual to the first-order part, the G\"odelization operator, defined below. Our starting point is Dzhafarov, Solomon, Yokoyama \cite[Theorem 1.6]{DSY23} which asserts that a problem is computably true (i.e., for every instance $p$ there exists a solution $q$ with $q \leq_\mathrm{T} p$) if and only if it is Weihrauch reducible to some first-order problem. The construction in their proof shows even more:

\begin{theorem} \thlabel{thm:min_first_order_above}
If $g$ is a problem which is computably true, then
\[ \min_{\leqW}\nolimits\{f \geqW g \st f \text{ is first-order}\} \]
exists and is represented by
\[ g^{\mathsf{G}}(p) = \{e \in \mathbb{N} \st \Phi_e(p) \in g(p)\}. \]
\end{theorem}
\begin{proof}
Notice $\dom(g^{\mathsf{G}}) = \dom(g)$ since $g$ is computably true. It is clear that $g \leqW g^{\mathsf{G}}$. Next, suppose $h$ has codomain $\mathbb{N}$ and $g \leqW h$ via $\Gamma$ and $\Delta$. To reduce $g^{\mathsf{G}}$ to $h$, observe that if $p \in \dom(g^{\mathsf{G}})$, then $\Gamma(p) \in \dom(h)$, and if $n \in h(\Gamma(p))$, then $\Delta(p,n) \in g(p)$. Using an index for $\Delta$, we can define a functional $\Psi$ such that $\Phi_{\Psi(p,n)}(p) = \Delta(p,n)$ for every $p \in \Baire$ and $n \in \mathbb{N}$. Then if $n \in h(\Gamma(p))$, we have $\Psi(p,n) \in g^{\mathsf{G}}(p)$ as desired. (Indeed $\Psi$ does not even need access to $p$, so $g^{\mathsf{G}} \leqsW h$.)
\end{proof}

Therefore, every problem $g$ which is computably true but not first-order lies in an interval with no first-order problems, bounded below by its first-order part $\fop{g}$ and above by $g^{\mathsf{G}}$. We choose to call $g^{\mathsf{G}}$ the \emph{G\"odelization} of $g$, inspired by the G\"odel numbering problem $\mathsf{G}$ studied by Brattka \cite{brattka_godel}.

\begin{proposition} \thlabel{prop:sufficient_fop_Det_below_Det_fop}
If $g$ is computably true, single-valued, and $g^{\mathsf{G}}$ is not deterministic (i.e., Weihrauch equivalent to a single-valued function with codomain $\Baire$), then $\fop{\Det{g^{\mathsf{G}}}} <_\mathrm{W} \Det{g^{\mathsf{G}}}$.
\end{proposition}
\begin{proof}
Since $g$ is single-valued and $g \leqW g^{\mathsf{G}}$, we have $g \leqW \Det{g^{\mathsf{G}}} <_\mathrm{W} g^{\mathsf{G}}$. So $\Det{g^{\mathsf{G}}}$ cannot be first-order, by \thref{thm:min_first_order_above}. We conclude that $\fop{\Det{g^{\mathsf{G}}}} <_\mathrm{W} \Det{g^{\mathsf{G}}}$.
\end{proof}

In \thref{prop:comp_true_single_valued_g^1_not_det}, we shall construct such $g$ using the following two ingredients.

\begin{proposition} \thlabel{prop:det_eq_realizer}
If $f$ is deterministic, then it is Weihrauch equivalent to one of its realizers.
\end{proposition}
\begin{proof}
Suppose $h : \subseteq \Baire \to \Baire$ satisfies $f \equivW h$. Fix $\Phi$ and $\Psi$ witnessing that $f \leqW h$. We define $r : \subseteq \Baire \to \Baire$ by $r(x) = \Psi(x,h(\Phi(x)))$. By construction, $r$ is a realizer of $f$, thus it in particular holds that $f \leqW r$. Also by construction, $r \leqW h$. Since $h \equivW f$, it follows that $r \equivW f$.
\end{proof}

The second ingredient is the following Rice-like result, which forms the recursion-theoretic core of our construction. For each $e \in \mathbb{N}$, let $\phi_e$ denote the $e$-th partial computable function on $\mathbb{N}$. Let $\mathrm{Tot}$ denote the set of $e \in \mathbb{N}$ such that $\phi_e$ is total.

\begin{lemma} \thlabel{lem:not_separable}
There is no computable $F: \subseteq \mathbb{N} \to \mathbb{N}$ such that whenever $e,e' \in \mathrm{Tot}$, then $F(e) = F(e')$ if and only if $\phi_e = \phi_{e'}$. In fact, fix $e_0 \in \mathrm{Tot}$ and define $A := \{e \in \mathbb{N} \st \phi_e = \phi_{e_0}\}$. Then no c.e.\ set containing $A$ is disjoint from $\mathrm{Tot}\setminus A$.
\end{lemma}
\begin{proof}
Suppose towards a contradiction that $U \supseteq A$ is c.e.\ and disjoint from $\mathrm{Tot}\setminus A$. Then we can compute halting as follows. Given $e$, compute an index $f(e)$ such that $\phi_{f(e)}$ copies $\phi_{e_0}$ until $\phi_e(e)$ halts (if ever), after which $\phi_{f(e)}$ flips each output of $\phi_{e_0}$. Notice if $\phi_e(e)$ halts, then $f(e) \in \mathrm{Tot}\setminus A$, else $f(e) \in A \subseteq U$. So we can compute whether $\phi_e(e)$ halts by waiting for it to halt, or for $f(e)$ to appear in $U$.
\end{proof}

\begin{proposition} \thlabel{prop:comp_true_single_valued_g^1_not_det}
There is a problem $g\function{\mathbb{N}}{\Baire}$ which is computably true and single-valued, such that $g^{\mathsf{G}}$ is not deterministic. (In particular, $g$ cannot be first-order, else $g^{\mathsf{G}} \equivW g$ is deterministic.)
\end{proposition}
\begin{proof}
We shall construct a sequence of computable reals $(g_n)_{n \in \mathbb{N}}$ such that for every pair of Turing functionals $\Gamma$ and $\Delta$ (which form a potential Weihrauch reduction from some realizer of $g^{\mathsf{G}}$ to $g^{\mathsf{G}}$ itself), one of the following holds:
\begin{enumerate}
	\item some $\Gamma(n)$ is not a name for a natural number (i.e., an instance of $g$ as defined below)
	\item \label{cond:splitting} there are $n$, $e$, $e' \in \mathbb{N}$ such that $\Phi_e(\Gamma(n)) = g_{\Gamma(n)} = \Phi_{e'}(\Gamma(n))$ but $\Delta(n,e) \neq \Delta(n,e')$
	\item \label{cond:wrong} there are $n$, $e \in \mathbb{N}$ such that $\Phi_e(\Gamma(n)) = g_{\Gamma(n)}$ but $\Phi_{\Delta(n,e)}(n) \neq g_n$.
\end{enumerate}

If we then define $g\function{\mathbb{N}}{\Baire}$ by $g(n) := g_n$, we see that $g$ is single-valued, computably true, and $g^{\mathsf{G}}$ is not Weihrauch equivalent to any of its realizers (hence not deterministic, by \thref{prop:det_eq_realizer}).

We shall construct $(g_n)_{n \in \mathbb{N}}$ in stages. At each stage, we handle a single pair $\Gamma$ and $\Delta$ by defining at most two new computable reals $g_n$. We assume that for every $n \in \mathbb{N}$, $\Gamma(n)$ names a natural number (otherwise no action is needed). Consider the following cases.

\underline{Case 1.} For some $n \in \mathbb{N}$ such that $\Gamma(n) \neq n$, $g_n$ has not been defined. By defining $g_{\Gamma(n)}$ if necessary, we may fix some $e$ such that $\Phi_e(\Gamma(n)) = g_{\Gamma(n)}$. Since $\Gamma(n) \neq n$, we may then define $g_n$ to be a computable real which differs from $\Phi_{\Delta(n,e)}(n)$ (regardless of whether the latter is total), ensuring (\ref{cond:wrong}).

\underline{Case 2.} Fix some $n \in \mathbb{N}$ such that $g_n$ is not defined. Since Case 1 fails, $\Gamma(n) = n$.

\underline{Case 2a.} There are $e,e' \in \mathbb{N}$ such that $\Phi_e(\Gamma(n)) = \Phi_{e'}(\Gamma(n))$ and $\Delta(n,e) \neq \Delta(n,e')$ (this includes the case where either or both sides diverge). We define $g_{\Gamma(n)} = \Phi_e(\Gamma(n))$, ensuring (\ref{cond:splitting}).

\underline{Case 2b.} If there is some $e \in \mathbb{N}$ such that $\Phi_e(\Gamma(n))$ is total and not equal to $\Phi_{\Delta(n,e)}(n)$, we define $g_n = \Phi_e(\Gamma(n))$. This ensures (\ref{cond:wrong}).

To complete the current stage of the construction, we shall show that the above cases encompass all possibilities. Suppose otherwise. Since Case 2b fails, for all $e \in \mathbb{N}$, whenever $\Phi_e(\Gamma(n))$ is total, so is $\Phi_{\Delta(n,e)}(n)$ and they must be equal. Since Case 2a fails, whenever $\Phi_e(\Gamma(n)) = \Phi_{e'}(\Gamma(n))$, we have $\Delta(n,e) = \Delta(n,e')$. We deduce that whenever $\Phi_e(\Gamma(n))$ and $\Phi_{e'}(\Gamma(n))$ are total, $\Delta(n,e) = \Delta(n,e')$ if and only if $\Phi_e(\Gamma(n)) = \Phi_{e'}(\Gamma(n))$. This contradicts \thref{lem:not_separable}: Define $F: \subseteq \mathbb{N} \to \mathbb{N}$ by $F(d) = \Delta(n,e)$, where $\Phi_e$ upon input $\Gamma(n)$ merely simulates $\phi_d$.

This completes the current stage of the construction. Observe we have ensured that one of (1)--(3) hold for $\Gamma$ and $\Delta$.
\end{proof}

We may now answer \cite[Question 3.5]{goh-pauly-valenti} in the affirmative.

\begin{theorem} \thlabel{thm:fop_Det_below_Det_fop}
There is a problem $f$ such that $\fop{\Det{f}} <_\mathrm{W} \Det{\fop{f}}$.
\end{theorem}
\begin{proof}
By \thref{prop:comp_true_single_valued_g^1_not_det}, fix a problem $g$ which is computably true and single-valued, such that $g^{\mathsf{G}}$ is not deterministic. By \thref{prop:sufficient_fop_Det_below_Det_fop}, $\fop{\Det{g^{\mathsf{G}}}} <_\mathrm{W} \Det{g^{\mathsf{G}}}$. Let $f = g^{\mathsf{G}}$. Then $\fop{\Det{f}} \equivW \fop{\Det{g^{\mathsf{G}}}} <_\mathrm{W} \Det{g^{\mathsf{G}}} \equivW \Det{\fop{f}}$, where the last equivalence holds because $f$ is first-order and $\Det{\cdot}$ is degree-theoretic.
\end{proof}

\section{$\pitacc{k}$ and $\pitaccN$}
\label{sec:pitaccn}
A crucial role in the separation between $\BS$ and $\DS$ in Section \ref{sec:separation} is played by partial orders that admit a tree decomposition. It is therefore natural to ask what is the strength of the restriction $\BStree$ of $\BS$ to trees. In order to present our results about the first-order and finitary parts of $\BStree$ we first prove some results about $\boldfacePi^0_2$ \emph{all or co-unique choice}, $\pitaccN$. This problem has also been studied in \cite{paulysolda} in the context of continuous Weihrauch reducibility.

\begin{definition}
The problem $\pitacc{k}$ receives as input a $\boldfacePi^0_2$-subset $A \subseteq \mathbf{k}$ with $|A| \geq k-1$ and returns some $i \in A$. The problem $\pitaccN$ receives as input a $\boldfacePi^0_2$-subset $A \subseteq \mathbb{N}$ with $|\mathbb{N} \setminus A| \leq 1$ and returns some $i \in A$.
\end{definition}

The relevance of these problems to $\BStree$ will be explained in Section \ref{sec:antichain}.

Rather than having to reason directly with $\boldfacePi^0_2$-subsets, we can use the following more concrete characterization of the problem.

\begin{proposition}
	\thlabel{thm:pi02acck}
Let $\mathbf{X} \in \{k, \mathbb{N}\}$ for $k \geq 2$. Then $\pitacc{\mathbf{X}}$ is equivalent to the problem ``Given some $p \in \mathbf{X}^\mathbb{N}$, output some $n \in\mathbf{X}$ such that $n \neq \lim_{i \to \infty} p_i$ (which we understand to be true if the limit does not exist)''.
\end{proposition}
\begin{proof}
Given some $p \in \mathbf{X}^\mathbb{N}$, we observe that $\{n \in \mathbf{X} \st n \neq \lim_{i \to \infty} p_i\}$ is a $\Pi^{0,p}_2$-set omitting at most one element. That shows that $\pitacc{\mathbf{X}}$ can solve the problem of finding a non-limit point.

Conversely, let $A$ be a $\boldfacePi^0_2$-subset of $\mathbf{X}$ with $|\mathbf{X}\setminus A|\le 1$. We can $A$-compute $q\function{\mathbf{X}\times \mathbb{N}}{2}$ such that, for each $n\in\mathbf{X}$, $q(n,\cdot)$ contains infinitely many $1$s iff $n \in A$. We define a sequence $p \in \mathbf{X}^\mathbb{N}$ as follows: for each $n\in \mathbf{X}$ and $k\in\mathbb{N}$, let $M_{n,k}:= |\{ j \le k \st q(n,j)=0 \}|$. 

For the sake of readability, we distinguish the cases where $\mathbf{X}$ is finite or not. If $\mathbf{X}=k$, we define  
\[ p_i:= \min \{ n < k \st (\forall m < k)(M_{n,i}\ge M_{m,i})\} . \]
If $\mathbf{X}=\mathbb{N}$, at each stage $i$, for each $n\le i$ we only see $j(n)$-many elements of the sequence $\sequence{q(n,j)}{j\in\mathbb{N}}$, where $j(n)$ is least such that $\pairing{n,j(n)}>i$. In this case, we define 
\[ p_i:= \min \{ n \le i \st (\forall m\le i)(M_{n,j(n)}\ge M_{m,j(m)})\} . \]

Observe that, in both cases, if $A = \mathbf{X} \setminus \{m\}$ for some $m$, then $\lim_{i\to\infty} p_i = m$, hence any $n\neq \lim_{i\to\infty} p_i$ is a valid solution for $\pitacc{\mathbf{X}}(A)$.
\end{proof}

Observe that $\pitacc{k}$ can be equivalently thought of as the restriction of $\RT{1}{k}$ to colourings where at most one color does not appear infinitely often. In particular, $\pitacc{2} \weiequiv \RT{1}{2}$.

While in some sense $\pitacc{\mathbb{N}}$ is a very weak principle -- after all, there is just one single incorrect answer amongst all of the natural numbers -- it is at the same time not particularly easy to solve, as evidenced by the following result.

\begin{proposition}
\thlabel{prop:pitacc_not_below_lim}
$\pitacc{\mathbb{N}} \nleqW \mflim$.
\end{proposition}
\begin{proof}
	As $\pitacc{\mathbb{N}}$ is a first-order problem, $\pitacc{\mathbb{N}} \leqW \mflim$ would already imply $\pitacc{\mathbb{N}} \leqW \C_\mathbb{N}$. The characterization in \thref{thm:pi02acck} shows that $\pitacc{\mathbb{N}}$ is a closed fractal. By the absorption theorem for closed fractals \cite[Theorem 2.4]{paulyleroux}, this means that $\pitacc{\mathbb{N}} \leqW \C_\mathbb{N}$ in turn would imply that $\pitacc{\mathbb{N}}$ is computable, which is readily seen to be false.
\end{proof}

The diamond operator \cite[Definition 9]{NP18} roughly captures the possibility of using a multi-valued function as oracle an arbitrary but finite number of times during a computation (with the additional requirement of having to declare, at some finite stage, that no more oracle calls will be made). It can be equivalently defined in terms of the following reduction game.

\begin{definition}[{\cite[Definitions 4.1, 4.3]{HJ16}, see also \cite[Definition 6.1]{soldavalenti}}]
	Let $f,g\pmfunction{\Baire}{\Baire}$ be two partial multi-valued functions. We define the reduction game $G(f\to g)$ as the following two-player game: on the first move, Player 1 plays $x_0\in \dom(g)$, and Player 2 either plays an $x_0$-computable $y_0\in g(x_0)$ and declares victory, or responds with an $x_0$-computable instance $z_1$ of $f$.
	
	For $n > 1$, on the $n$-th move (if the game has not yet ended), Player 1 plays a solution $x_{n-1}$ to the input $z_{n-1}\in\dom(f)$. Then Player 2 either plays a $\pairing{x_0,\hdots,x_{n-1}}$-computable solution to $x_0$ and declares victory, or plays a $\pairing{x_0,\hdots,x_{n-1}}$-computable instance $z_n$ of $f$.
	
	If at any point one of the players does not have a legal move, then the game ends with a victory for the other player. Player 2 wins if it ever declares victory (or if Player 1 has no legal move at some point in the game). Otherwise, Player 1 wins.
	
	We define $f^\diamond\pmfunction{\mathbb{N}\times\Baire}{(\Baire)^{<\mathbb{N}}}$ as the following problem:
	\begin{itemize}
		\item $\dom(f^\diamond)$ is the set of pairs $(e,d)$ s.t.\ Player 2 wins the game $G(f\to \id)$ when Player 1 plays $d$ as his first move, and $\Phi_e$ is a winning strategy for Player 2;
		\item a solution is the list of moves of Player $1$ for a run of the game.
	\end{itemize}

	Observe that $g \weireducible f^\diamond$ iff Player 2 has a computable winning strategy for the game $G(f\to g)$, i.e.\ if there is a Turing functional $\Phi$ such that Player 2 always plays $\Phi(x_0, \pairing{x_1,\hdots,x_{n-1}})$, and wins independently of the strategy of Player 1. 
	
	We described the game assuming that $f,g$ have domain and codomain $\Baire$. The definition can be extended to arbitrary multi-valued functions, and the moves of the players are \emph{names} for the instances/solutions.
\end{definition}

\begin{proposition}
	\thlabel{prop:binarypart}
	$\Fin{k}{\pitacc{k+2}^\diamond} \equivW \id$ for every $k$.
\end{proposition}
\begin{proof}
	Let $g\pmfunction{\Baire}{k}$ be such that $g\weireducible\pitacc{k+2}^\diamond$ and let $\Phi$ be the computable functional witnessing the winning strategy for Player 2 for the game $G= G(\pitacc{k+2}\to g)$. For each input $x\in\dom(g)$, we can build the c.e.\ tree $T$ of all possible runs of the game $G$ where Player 1 starts by playing $x$. More precisely, for $\sigma\in (k+2)^{<\mathbb{N}}$ and $i<k+2$, we enumerate $\sigma\concat i \in T$ if $\sigma \in T$ and Player 2 does not declare victory on the $(\length{\sigma}+1)$-th move. Moreover, if Player 2 declares victory on the $(\length{\sigma}+1)$-th move, we enumerate $\sigma\concat \pairing{k+2, \Phi(x, \sigma)}$ in $T$ and commit to never extend this branch further. Intuitively, if Player 2 declares victory then we reached a leaf of the tree, and the leaf has a special label that contains the last Player 2's move, i.e.\ the solution to $g(x)$ obtained using $\Phi$ when the oracle answers are $\sigma(0),\hdots,\sigma(\length{\sigma}-1)$.

	Observe that, each non-leaf $\sigma \in T$ has exactly $k+2$-many children, and at least $k+1$ of them correspond to valid Player 1 moves (i.e.\ possible $\pitacc{k+2}$-answers). Observe also that the tree $T$ need not be well-founded and not all the leaves are necessarily labelled as such. Indeed, the hypothesis that Player 2 has a winning strategy only guarantees that the subtree $S\subseteq T$ corresponding to valid runs of the game is well-founded (and it is computable to tell if a string in $S$ is a label). 

	By unbounded search, we can compute a well-founded tree $S\subseteq T$ such that, $\varepsilon\in S$, each non-leaf $\sigma\in S$ has exactly $(k+1)$-many children, and each leaf of $S$ is of the form $\tau\concat \pairing{k+2, \Phi(x, \tau)}$ for some $\tau\in (k+2)^{<\mathbb{N}}$ and $\Phi(x, \tau)<k$. To compute a solution for $g(x)$ we proceed recursively on the (finite) rank of $S$: each leaf $\tau\concat \pairing{k+2, \Phi(x, \tau)}$ of $S$ is simply labelled with $\Phi(x, \tau)$. To compute the label of a non-leaf $\sigma$, observe that, by the pigeonhole principle, we can choose $i<k$ such that at least two of the $k+1$ children of $\sigma$ are labelled with $i$. We can then label $\sigma$ with $i$ and move to the next stage. 

	We claim that the label of $\varepsilon$ is a valid solution for $g(x)$. Indeed, assume $\tau\in S$ corresponds to a valid run of the game, namely for each $i<\length{\tau}$, $\tau(i)$ is a valid $(n+2)$-th move for Player 1 when Player 2 plays according to $\Phi$ and the first $(n+1)$ moves of Player 1 are $x, \tau(0),\hdots, \tau(i-1)$. By induction, any such $\tau$ receives a label $i<k$ such that $i\in g(x)$. The statement follows from the fact that $\varepsilon$ always represents a valid run of the game.
\end{proof}

Note that the previous proof only uses that $\pitacc{k+2}$ is a problem with codomain $k+2$ having at most one incorrect answer for every input. Therefore, for every pointclass $\boldsymbol{\Gamma}\in \{\boldfaceSigma^0_k,\boldfacePi^0_k,\boldfaceDelta^0_k, \boldfaceSigma^1_1,\boldfacePi^1_1,\boldfaceDelta^1_1 \}$, we immediately obtain:

\begin{corollary}
	\thlabel{thm:fin_k_acck+2_computable}
$\Fin{k}{\parallelization{\boldsymbol{\Gamma}\mathsf{-ACC}_{k+2}}} \weiequiv \id$ for every $k\ge 1$.
\end{corollary}
\begin{proof}
	This can be proved observing that, for every $f$,  $\Fin{k}{f} \weireducible \fop{f}$, and therefore $\Fin{k}{f} \weiequiv \Fin{k}{\fop{f}}$, as $\Fin{k}{f}$ is the strongest problem with codomain $k$ that is reducible to $f$. This implies that  
	\[ \Fin{k}{\parallelization{\boldsymbol{\Gamma}\mathsf{-ACC}_{k+2}}} \weiequiv \Fin{k}{\fop{\left(\parallelization{\boldsymbol{\Gamma}\mathsf{-ACC}_{k+2}} \right) }} \weireducible \Fin{k}{\boldsymbol{\Gamma}\mathsf{-ACC}_{k+2}^\diamond }\weireducible \id ~,  \]
	where the second inequality follows from the fact, for every first-order $f$, $\fop{\parallelization{f}}\weireducible f^\diamond$ (see \cite[Theorem 5.7 and Proposition 6.3]{soldavalenti}). The last inequality is the statement of \thref{prop:binarypart}.	
\end{proof}

\begin{corollary}
	\thlabel{thm:det_acck+2_computable}
$\Det{\parallelization{\boldsymbol{\Gamma}\mathsf{-ACC}_{k+2}}} \equivW \id$ for every $k\ge 2$.
\end{corollary}
\begin{proof}
	\thref{prop:Det_below_parallel_Det_2} and the definitions of deterministic/$k$-finitary part imply that, for every $f$ and every $k\ge 2$,
	\[ \Det{f} \weireducible \parallelization{\DetPartX{k}{f}} \leqW \parallelization{\Fin{k}{f}}. \]
	Letting $f= \parallelization{\boldsymbol{\Gamma}\mathsf{-ACC}_{k+2}}$, the claim follows from \thref{thm:fin_k_acck+2_computable}.
\end{proof}

We observe that \thref{prop:binarypart} cannot be improved to show that $\Fin{k}{\pitacc{k+1}^\diamond}$ is computable.

\begin{proposition}
\thlabel{prop:acckbelowpi02acckplus1}
	$\ACC_{k} \weireducible \Fin{k}{\pitacc{k+1}}$ for every $k\ge 2$.
\end{proposition}
\begin{proof}
	We use the equivalent formulation of $\pitacc{k+1}$ introduced in \thref{thm:pi02acck}.	For every $A\in\dom(\ACC_k)$, we uniformly compute a sequence $p\in\Baire$ as follows: 
	\[ p(s):=\begin{cases}
		i & \text{if } i \text{ is enumerated outside of }A \text{ by stage }s \\
		k & \text{otherwise.}
	\end{cases} \]
	Let $n<k+1$ be such that $n\neq \lim_{s\to\infty} p(s)$. If $n=k$ then $A\neq k$, therefore we can computably find (by unbounded search) the unique $m \in k\setminus A$. Conversely, if $n\neq k$ then it follows by definition of $p$ that $n\in A$. 
\end{proof}

The following proposition shows that for $k > 2$, the degree $\Fin{k}{\pitacc{k+1}}$ escapes a ``nice'' characterization. The exceptional case of $\Fin{2}{\pitacc{3}}$ is covered in \thref{corr:fin2pitcacc3} below.

\begin{proposition}
\thlabel{prop:finkpi02acckplus1}
	Let $\mathbb{N}_\bot := \mathbb{N}\cup \{\bot\}$ be the represented space where $\bot$ is represented by $0^\omega$ and every $n\in\mathbb{N}$ is represented by any $p\neq 0^\omega$ such that $n+1$ is the first non-zero element of $p$. 

	For every $k\ge 2$, $\Fin{k}{\pitacc{k+1}}$ is equivalent to the problem $F$ defined as follows: given $A\in \dom(\pitacc{k})$ and $n\in\mathbb{N}_\bot$ with the assumption that
	\begin{itemize}
		\item $n=\bot$ iff $|A| = k$
		\item if $n\in\mathbb{N}$ then $n\in A$,
	\end{itemize}
	find $m\in \pitacc{k}(A)$.
\end{proposition}
\begin{proof}
    We shall use the characterization of $\pitacc{k}$ introduced in \thref{thm:pi02acck}, i.e.\ we can assume that a name for $A\in \dom(\pitacc{k})$ is a sequence $p\in k^\mathbb{N}$ such that if $\lim_n p(n)$ exists then it is the unique $i<k$ not in $A$.

	Let us first show that $F\weireducible \pitacc{k+1}$.  Let $\pairing{p, q}$ be a name for $(A,n)\in \dom(F)$. We define a sequence $r\in (k+1)^\mathbb{N}$ as follows: as long as $q(i)=0$, let $r(i):=k$. If $q(i)\neq 0$ for some $i$, then, for every $j\ge i$, $r(j):=p(j)$. 

	Clearly, $r$ is a (name for a) valid input for $\pitacc{k+1}$. Observe that if $q$ is not constantly $0$ then $|A|<k$, hence the sequence $p$ has a limit $l<k$. In particular, $r$ always has a limit. We claim that, given $m\neq \lim_{i\to\infty} r(i)$, we can uniformly compute a valid solution for $F(A,n)$. Indeed, if $m<k$, then $m\in A$. This follows from the fact that if $q$ is constantly $0$ then $|A|=k$. Conversely, if $q$ is not constantly $0$ then $m\neq\lim_i p(i)$, and therefore $m\in A$. On the other hand, if $m=k$ then $q$ is not constantly $0$, so, by hypothesis, if $q(i)$ is the first non-zero element of $q$ then $q(i)-1\in A$. This shows that $F\weireducible \pitacc{k+1}$.

	Assume now that $f\pmfunction{\Baire}{k}$ is such that $f\weireducible \pitacc{k+1}$ via $\Phi,\Psi$. For the sake of readability, we can assume that $\Psi$ has codomain $\mathbb{N}$. To reduce $f$ to $F$, suppose we are given some $x \in \dom(f)$. In parallel, we can simulate the computations $\Psi(x,i)$ for every $i<k+1$. We can uniformly compute two sequences $p, q\in k^\mathbb{N}$ as follows: we wait until we either find $i<j<k+1$ such that $\Psi(x,i)=\Psi(x,j)=n$ or we see that $|\{ \Psi(x,i) \st i<k+1\}|=k$ (by the pigeonhole principle, at least one of two cases must happen). If we first find $i<j<k+1$ such that $\Psi(x,i)=\Psi(x,j)=n$, we define $p,q$ so that $0^s (n+1)\prefix q$ for some $s$ and $p$ is any convergent sequence with $\lim p \neq n$. On the other hand, if we first see that $|\{ \Psi(x,i) \st i<k+1\}|=k$, we proceed as follows: we keep searching for $i<j$ as above. If they are never found, we let $q:=0^\omega$ and $p$ be any non-convergent sequence. Conversely, if they are found, we define $q$ so that $0^s (n+1)\prefix q$ for some $s$. Moreover, we let $p$ be a convergent sequence with $\lim p\neq n$ such that if $\lim_t \Psi(x,\Phi(x)(t))$ exists and is different from $n$ then $\lim p = \lim_t \Psi(x,\Phi(x)(t))$. This can be done by computing $\Psi(x,\Phi(x)(t))$ and replacing each occurrence of $n$ with some $m\neq n$.

	Observe that $\pairing{p,q}$ is a valid name for an input $(A,n_\bot)$ of $F$. Indeed, $q$ is constantly $0$ iff $p$ is not convergent, i.e.\ iff $p$ is a name for the input $A$ of $\pitacc{k}$ with $|A|=k$. Moreover, if $q$ is not constantly $0$ then it is a name for some $n\in \mathbb{N}$ with $n\neq \lim p$. Given $m\in F(A,n_\bot)$, we can uniformly compute a solution for $f(x)$ as follows: we wait until we see that $q$ is the name for some $n\neq \bot$ or that $|\{ \Psi(x,i) \st i<k+1\}|=k$. If the former is observed first, then $n\in f(x)$, as there are $i<j$ with $\Psi(x,i)=\Psi(x,j)=n$, and either $i$ or $j$ is a valid solution for $\pitacc{k+1}(\Phi(x))$. Conversely, if we first see that $|\{ \Psi(x,i) \st i<k+1\}|=k$, then any $m\in F(A,n_\bot)$ is a valid solution for $f(x)$. Indeed, if $|f(x)|<k$ then the (unique) wrong solution is $\lim \Psi(x,\Phi(x))$ (as $\pitacc{k+1}(\Phi(x))$ has at most one wrong solution). By construction, $m\neq \lim p = \lim \Psi(x,\Phi(x))$, and therefore $m\in f(x)$. Observe that, in this case, $q$ is not the constantly $0$ sequence (if $q$ is constantly $0$ then every $i<k+1$ is valid solution for $f(x)$). 
\end{proof}

\begin{corollary}
\thlabel{corr:fin2pitcacc3}
$\Fin{2}{\pitacc{3}}\equivW \C_2$.
\begin{proof}
As $\ACC_2 \equivW \C_2$, one direction follows from \thref{prop:acckbelowpi02acckplus1}. That $\Fin{2}{\pitacc{3}} \leqW \C_2$ follows from \thref{prop:finkpi02acckplus1} by observing $\C_2$ can solve the task ``Given $n \in \mathbb{N}_\bot$, compute some $b \in \{0,1\}$ such that if $n \in \{0,1\}$, then $b = n$''.
\end{proof}
\end{corollary}

\section{Finding antichains in trees}
\label{sec:antichain}

We turn our attention to trees as a special case of partial orders, where we identify a tree with its prefix relation. We observe that a tree is a wqo iff it has finite width; in particular, being wqo is merely a $\Sigma^0_2$-property for trees rather than $\Pi^1_1$ as for general quasi-orders. While trees can only avoid being wqo by having an infinite antichain, being a bad sequence is a weaker notion than being an antichain. However, a given infinite bad sequence in a tree can be computably thinned out to yield an infinite antichain. We let $\BStree$ denote the restriction of $\BS$ to subtrees of $\baire$, and think of it as the problem that receives as input a tree of infinite width and returns an infinite antichain in it. How $\BStree$ and its first-order part relate to other principles, in particular those discussed in the previous section, is depicted in Figure \ref{fig:summary:bstree}.

\begin{figure}[htb]
	\begin{center}
		\def\angle{90}
		\def\dist{2.2}
		\def\chainstretch{1.7}
		\begin{tikzpicture}[scale=0.8, node distance=10mm and 5mm, on grid=true]
	
			\node (ACCN) {$\ACC_\mathbb{N}$};
		
			\path (ACCN) 
				++ (\angle:\chainstretch*\dist) node (ACCk1) {$\ACC_{k+1}$}%
				++ (\angle:\dist) node (ACCk) {$\ACC_{k}$}%
				++ (\angle:\chainstretch*\dist) node (ACC2) {$\ACC_2 \weiequiv \Choice{2}$};

			\node [above right=1.7cm and 3.3cm of ACCN] (pACCN) {$\pitaccN$};
			\path (pACCN) 
				++ (\angle:\chainstretch*\dist/2) node (BST1) {$\fop{\BStree}$}%
				++ (\angle:\chainstretch*\dist/2) node (pACCk1) {$\pitacc{k+1}$}%
				++ (\angle:\dist) node (pACCk) {$\pitacc{k}$}%
				++ (\angle:\chainstretch*\dist) node (pACC2) {$\pitacc{2} \weiequiv \RT{1}{2}$};

			\node [above right=0cm and 3.3cm of pACCN] (parACCN) {$\parallelization{\ACC_\mathbb{N}}$};
			\path (parACCN) 
				++ (\angle:\chainstretch*\dist) node (parACCk1) {$\parallelization{\ACC_{k+1}}$}%
				++ (\angle:\dist) node (parACCk) {$\parallelization{\ACC_{k}}$}%
				++ (\angle:\chainstretch*\dist) node (parACC2) {$\parallelization{\ACC_{2}}\weiequiv \Choice{\Cantor}$};
	
			\node [above right=1.7cm and 3.3cm of parACCN] (parpACCN) {$\parallelization{\pitaccN}$};
			\path (parpACCN) 
				++ (\angle:\chainstretch*\dist) node (parpACCk1) {$\parallelization{\pitacc{k+1}}$}%
				++ (\angle:\dist) node (parpACCk) {$\parallelization{\pitacc{k}}$}%
				++ (\angle:\chainstretch*\dist) node (parpACC2) {$\parallelization{\pitacc{2}}\weiequiv \KL$};
		
			\node [above right=1.4cm and 5cm of BST1] (BST) {$\BStree$};
	
			%arrows
			\arrowchain{ACCN}{ACCk1}{0.33}{0.66}
			\draw[->] (ACCk1) -- (ACCk);
			\arrowchain{ACCk}{ACC2}{0.33}{0.66}
	
			\draw[->] (pACCN) -- (BST1);
			\arrowchain{BST1}{pACCk1}{0.33}{0.66}
			\draw[->] (pACCk1) -- (pACCk);
			\arrowchain{pACCk}{pACC2}{0.33}{0.66}
	
			\arrowchain{parACCN}{parACCk1}{0.33}{0.66}
			\draw[->] (parACCk1) -- (parACCk);
			\arrowchain{parACCk}{parACC2}{0.33}{0.66}
	
			\arrowchain{parpACCN}{parpACCk1}{0.33}{0.66}
			\draw[->] (parpACCk1) -- (parpACCk);
			\arrowchain{parpACCk}{parpACC2}{0.33}{0.66}
			
			\draw[->] (ACCN) -- (parACCN);
			\draw[->] (ACCk1) -- (parACCk1);
			\draw[->] (ACCk) -- (parACCk);
			\draw[->] (ACC2) -- (parACC2);
	
			\draw[->] (ACCN) -- (pACCN);
			\draw[->] (ACCk1) -- (pACCk1);
			\draw[->] (ACCk) -- (pACCk);
			\draw[->] (ACC2) -- (pACC2);
	
			\draw[->] (pACCN) -- (parpACCN);
			\draw[->] (pACCk1) -- (parpACCk1);
			\draw[->] (pACCk) -- (parpACCk);
			\draw[->] (pACC2) -- (parpACC2);
	
			\draw[->] (parACCN) -- (parpACCN);
			\draw[->] (parACCk1) -- (parpACCk1);
			\draw[->] (parACCk) -- (parpACCk);
			\draw[->] (parACC2) -- (parpACC2);

			\draw[->] (BST1) -- (BST);
			\arrowchain{BST}{parpACCk1}{0.33}{0.66}
		\end{tikzpicture}	

		\caption{The figure shows all reductions between the depicted principles up to transitivity, with the potential exception that we have not ruled out being $\BStree \leqW \parallelization{\pitaccN}$. The diagram can be thought of as a cube whose vertical edges have decreasing chains parameterized by $k \geq 2$ on them, and the bottom corners are labelled by degrees which are below the ones on the edges for all $k$.} 
	\label{fig:summary:bstree}
	\end{center}
\end{figure}

Even though the instances of $\BStree$ are arbitrary subtrees of $\baire$ of infinite width, we show below that its strength comes already from the pruned subtrees of $\cantor$ (a tree is \emph{pruned} if it has no leaves). Compare this with the problem $\C_{\Baire}$ of finding an infinite path through a given ill-founded subtree of $\baire$: Its restriction to ill-founded pruned trees is of course computable, and its restriction $\C_\Cantor$ to subtrees of $\cantor$ is known to be much weaker (see \cite[Figure 11.2]{bgp21}). This is unsurprising given the difference in the complexity between ``$T$ is wqo'' and ``$T$ is well-founded''.

\begin{proposition}
\thlabel{prop:treeagnostic}
$\BStree$ is Weihrauch equivalent to its restriction to pruned subtrees of $\cantor$.
\begin{proof}
Given a tree $T \subseteq \baire$ of infinite width, we shall uniformly compute a pruned tree $S \subseteq \cantor$, also of infinite width, together with a surjective labelling function $\ell\function{S}{T}$, such that if $\ell(v) \prefix \ell(w)$ then $v \prefix w$. The image of any infinite $S$-antichain under $\ell$ would then be an infinite $T$-antichain.

Our construction proceeds in phases, where at the end of phase $n$, we commit to $S \cap 2^{<n}$ and the corresponding restriction of $\ell$. Throughout the construction we will respect the constraint that $\ell(v) \prefix \ell(w)$ implies $v \prefix w$. At the end of phase $n$, all leaves in $S$ will have height at least $n$, so $S$ will be ultimately be a pruned tree.

During phase $0$, we simply add the root to $S$ and label it with the root of $T$. Fix an ordering of the vertices of $T$ (with order-type $\omega$) which agrees with $\prefix$. In phase $n$, start by considering the first vertex $u \in T$ that is not yet enumerated in $\ran(\ell)$. Our constraint on $\ell$ implies that the set $\{v' \in S: \ell(v') \prefix u\}$ forms a path in $S$, say with final vertex $v'$. (Observe that $\ell(v')$ is the parent of $u$.) Let $v$ be the first vertex of height at least $n-1$ such that $v' \prefix v$ and either $v\concat 0 \notin S$ or $v\concat 1 \notin S$. Such $v$ exists because all leaves in $S$ have height at least $n-1$. We then add a child to $v$ (i.e., put $v\concat 0$ in $S$ if $v\concat 0 \notin S$, else put $v\concat 1$ in $S$) and label the child with $u$.

If we now have both $v\concat 0$ and $v\concat 1$ in $S$, we conclude phase $n$ by adding a child to each vertex in $S$ which was in $S$ at the beginning of phase $n$ and is currently a leaf in $S$. Each such child is labelled the same way as its parent in $S$. This ensures every leaf in $S$ has height at least $n$. Notice this action maintains our constraint on $\ell$.

On the other hand, if only one of $v\concat 0$ and $v\concat 1$ lie in $S$, we consider the next $u \in T$ that is not in $\ran(\ell)$ and repeat the above steps. To see why this process terminates, observe that if it did not, then there must be some $u \in T$ first considered in phase $n$ which has multiple children. Suppose we acted for $u$ by adding $v\concat 0$ to $S$ and defining $\ell(v\concat 0) = u$. Let $u'$ denote the first child of $u$ to be considered in phase $n$. Then $v\concat 0$ would be the first vertex of height at least $n-1$ which extends every $v'$ with $\ell(v') \prefix u'$, because the highest vertex $v'$ with $\ell(v') \prefix u'$ is precisely $v\concat 0$, which has height at least $n$. Therefore we would add $v\concat 00$ to $S$ and define $\ell(v\concat 00) = u'$. Similarly, if $u''$ denotes the second child of $u$ to be considered in phase $n$, then we would add $v\concat 01$ to $S$ and define $\ell(v\concat 01) = u''$. We would then conclude phase $n$, contrary to assumption.

This completes the description of our construction. We have ensured that $S$ contains infinitely many vertices $v$ with $v\concat 0,v\concat 1 \in S$, so $S$ has infinite width. The other requirements are directly enforced by the construction.
\end{proof}
\end{proposition}

Next, we shall characterize the finitary parts and deterministic part of $\BStree$. If $(T,\sqsubseteq)$ is a tree and $v$ is a vertex in $T$, define $[v]_T = \{w \in T: v \sqsubseteq w\}$. We omit the subscript $T$ if the tree in question is clear from context.

\begin{proposition}
\thlabel{prop:bswtrupperbound}
$\BStree \leqW \parallelization{\pitacc{k}}$ for every $k \geq 2$.
\begin{proof}
Given a tree $T \subseteq \baire$ with infinite width, we consider all antichains $(v_0,v_1,\ldots,v_{k-1})$ of size $k$ in it. For each such antichain, at most one vertex $v_i$ must be avoided when building an infinite antichain, which happens if $T\setminus[v_i]$ has finite width. We can use $\pitacc{k}$ to select a suitable choice amongst the $v_i$. We can then greedily build an infinite antichain from the selected vertices.
\end{proof}
\end{proposition}

\begin{corollary}
	\thlabel{cor:fin_det_bstree}
$\Fin{k}{\BStree} \equivW \id$ for every $k\ge 1$. Therefore, $\Det{\BStree} \equivW \id$.
\end{corollary}
\begin{proof}
	For the first part of the statement, fix $k\ge 1$. By \thref{prop:bswtrupperbound} and \thref{thm:fin_k_acck+2_computable}, 
	\[ \Fin{k}{\BStree}\weireducible \Fin{k}{\parallelization{\pitacc{k+2}}} \weiequiv \id. \] 
	Similarly, the second part of the statement then follows from $\Det{\parallelization{\pitacc{k+2}}} \equivW \id$ (\thref{thm:det_acck+2_computable}).
\end{proof}

We now turn our attention to the first-order part of $\BStree$. The following first-order problem clearly reduces to $\BStree$.

\begin{definition}
Let $\ExtVer$ denote the problem taking a tree $T$ of infinite width, and returning a vertex $v \in T$ such that $T \setminus [v]$ still has infinite width, i.e., a vertex that is extendible to an infinite antichain.    
\end{definition}

The construction employed to prove \thref{prop:treeagnostic} also shows that for $\ExtVer$ it is immaterial whether we consider countably-branching or binary trees, and whether we assume the trees to be pruned or not.

\begin{proposition}
	\thlabel{thm:fop_bstree}
$\ExtVer \equivW \fop{\BStree}$.
\end{proposition}
\begin{proof}
We only need to show that $\fop{\BStree} \leqW \ExtVer$. An instance of $\fop{\BStree}$ is a tree $T$ of infinite width together with a notion of ``sufficiently large finite antichain'' (more formally, the latter is a computable functional which, when given $T$ and an infinite $T$-antichain, must halt on some initial segment thereof). A solution is a sufficiently large finite antichain which is extendible to an infinite one. We shall compute a tree $T' \subseteq \baire$ of infinite width such that each vertex $v \in T'$ (other than the root) is labelled with a sufficiently large finite antichain $A_v$ in $T$, and if $v$ extends to an infinite $T'$-antichain, then $A_v$ extends to an infinite $T$-antichain. For clarity, we shall use Greek letters ($\sigma$, $\tau$) to denote the vertices of $T$, and Roman letters ($v$, $w$) to denote the vertices of $T'$.

The tree $T'$ is defined recursively as follows. Start by putting the root in $T'$. The children of the root will each be labelled with a sufficiently large finite antichain $A_i$ in $T$. Specifically, the $i$-th child of the root is defined recursively as follows: First search for $A_0$ and $A_1$ in $T$, disjoint and each sufficiently large, such that $A_0 \cup A_1$ forms an antichain.
Since $T$ has infinite width, this search is successful and we label the first child with $A_0$ and the second child with $A_1$. Then search for $A_2$ sufficiently large and disjoint from $A_0 \cup A_1$ such that $A_0 \cup A_1 \cup A_2$ forms an antichain. If found, we label the third child with $A_2$. Proceed to search for $A_3$, and so on. Note that $\sequence{A_i}{i}$ may not be infinite.
The reason why we start by searching for two antichains $A_0$ and $A_1$, and then continue by searching one extra antichain at a time (first $A_2$, then $A_3$, and so on) is to guarantee that if a node has a child, then it has at least two incomparable children. 
More generally, if $v \in T'$ is not the root, then we define the children of $v$ by searching as above among the sufficiently large antichains $B$ such that $B \trianglerighteq^T A_v$ (i.e., there is some $\sigma \in A_v$ such that $\sigma \prefix \tau$ for all $\tau \in B$).
Note that for certain $v$, this search may turn up empty. This completes the construction of $T'$, excepting a final modification to ensure that $T'$ is computable (uniformly in $T$): Each child of a vertex in $T'$ shall encode the stage at which its corresponding $T$-antichain was found.

We claim that $T'$ has infinite width. If any vertex has infinitely many children, we are done. Otherwise, we shall prove by induction that at each level of $T'$ (beyond the root), there is some vertex $v$ and some $\sigma \in A_v$ such that $[\sigma]_T$ has infinite width. This would imply that $v$ has at least two children, allowing us to conclude that $T'$ has infinite width. The inductive step proceeds as follows: Suppose $v \in T'$ and $\sigma \in A_v$ are such that $[\sigma]_T$ has infinite width. Say the children of $v$ are labelled by the antichains $A_0,A_1,\dots,A_k$. Then $A_0 \cup \dots \cup A_k$ is not extendible to an infinite antichain in $\bigcup_{\rho \in A_v} [\rho ]_T$ (else $v$ would have additional children), but the latter contains $[\sigma]_T$ and thus has infinite width. So there is some $i$ and some $\rho \in A_i$ such that $[\rho]_T$ has infinite width. The base case proceeds similarly; simply consider the children of the root in $T'$ and use the assumption that $T$ has infinite width.

Next, we claim that if $v_0,v_1 \in T'$ are incomparable, then any $\sigma_0 \in A_{v_0}$ and any $\sigma_1 \in A_{v_1}$ are $T$-incomparable. This would imply $A_{v_0}$ and $A_{v_1}$ are disjoint and $A_{v_0} \cup A_{v_1}$ is a $T$-antichain. To prove the claim, let $w \in T'$ denote the longest common ancestor of $v_0$ and $v_1$. Let $w_0 \prefix v_0$ and $w_1 \prefix v_1$ be children of $w$. By construction, $A_{w_0}$ and $A_{w_1}$ are disjoint and $A_{w_0} \cup A_{w_1}$ is a $T$-antichain. Also by construction, for $i =0,1$, there is some $\tau_i \in A_{w_i}$ such that $\tau_i \prefix \sigma_i$. Since $\tau_0$ and $\tau_1$ are $T$-incomparable, so are $\sigma_0$ and $\sigma_1$.

The above claim implies that if $v \in T'$ is extendible to an infinite $T'$-antichain $\sequence{v_i}{i}$, then $A_v$ is extendible to an infinite $T$-antichain, namely $\bigcup_i A_{v_i}$. Since every $A_v$ is sufficiently large, this yields a Weihrauch reduction from $\fop{\BStree}$ to $\ExtVer$.
\end{proof}

Similar to \thref{prop:bswtrupperbound}, we shall obtain an upper bound for $\fop{\BStree}$. Let $\bigsqcap_{k \geq 2} \pitacc{k}$ denote the following problem: Given a sequence $\sequence{p_k}{k \in \mathbb{N}}$ where each $p_k$ is an $\pitacc{k}$-instance, find $(k,i) \in \mathbb{N}^2$ such that $i \in \pitacc{k}(p_k)$. Observe that $\bigsqcap_{k \geq 2} \pitacc{k} \leqW \ACC_\ell$ for every $\ell$. The operation $\sequence{f_n}{n\in\mathbb{N}} \mapsto \bigsqcap_{n\in\mathbb{N}} f_n$ is a generalization of the binary meet to countably many problems. However, it is known that the Weihrauch degrees are not a $\aleph_0$-complete meet semilattice (\cite{HiguchiPauly13}, see also \cite{LMV24}).

\begin{proposition}
\thlabel{prop:bstreefirstorderupperbound}
$\ExtVer \leW \bigsqcap_{k \geq 2} \pitacc{k}$.
\end{proposition}
\begin{proof}
First, we prove the reduction. Given a tree $T$ of infinite width, search for antichains of each finite size. Within each antichain, at most one vertex is not extendible to an infinite antichain in $T$. Since extendibility is a $\Pi^0_2$ property (relative to $T$), for each $k$ we can build an input for $\pitacc{k}$ to identify one vertex in the antichain of size $k$ which is extendible. Then a solution for any of the $\pitacc{k}$-instances yields an $\ExtVer$-solution to $T$.

To show that the reduction is strict, assume towards a contradiction that $\bigsqcap_{k \geq 2} \pitacc{k} \leqW \ExtVer$. Apply the reduction to the ``neutral'' input to $\bigsqcap_{k \geq 2} \pitacc{k}$, i.e., where every $j < k$ is a solution. If the resulting tree $T$ contains incomparable vertices $v_0$ and $v_1$ which yield the same output $j < k$ under the reduction, we may fix a finite prefix which determines this fact, then render $j$ wrong in the $\pitacc{k}$-instance by enumerating it. This leads to a contradiction because either $v_0$ or $v_1$ must be a valid $\ExtVer$-solution. On the other hand, if incomparable vertices in $T$ always yield different outputs under the reduction, by the pigeonhole principle, there must be incomparable vertices $v_0$ and $v_1$ which yield outputs for different $k$. Say $v_i$ determines the output $j_i < k_i$. Once again this fact is determined by a finite prefix of the neutral input. We then render $j_i$ wrong in the $\pitacc{k_i}$-instance for $i=0,1$ after said finite prefix. Since either $v_0$ or $v_1$ must be an $\ExtVer$-solution, we reach a contradiction.
\end{proof}

It turns out that $\ExtVer$ is sandwiched between $\pitaccN$ and $\bigsqcap_{k \geq 2} \pitacc{k}$.

\begin{proposition}
\thlabel{prop:pi02accnlowerbound}
$\pitaccN \leqW \ExtVer$.
\begin{proof}
Given some $p \in \Baire$, we shall build some $T \subseteq \cantor$ of infinite width such that from any vertex in $T$ which is extendible to an infinite antichain, we can compute some $n$ such that $p$ does not end in $n^\omega$ (i.e., if $\lim_i p_i$ exists, it is not $n$). We can assume without loss of generality that $p_i \leq i$. The tree is constructed level by level. After each level is constructed, each newly added vertex is labelled by the least natural number which has yet to be used as a label, in order from left to right. When building level $s$, first consider the leftmost leaf which extends the vertex with label $p_s$. Such a vertex must exist because $T$ already has more than $s \geq p_s$ vertices. Add both children of said leaf to $T$. Then, for every other leaf in $T$ (excluding the children just added), add its left child to $T$.

As we have a bifurcation on every level and no dead-ends, $T$ has infinite width. Moreover, if $p$ ends in $n^\omega$, then the subtree of $T$ not extending the vertex with label $n$ has finite width, i.e.~the vertex with label $n$ is not extendible.
\end{proof}
\end{proposition}

The following proposition is stated for $\pitaccN$ but can be easily generalized to any pointclass $\boldsymbol{\Gamma}$.

\begin{proposition}
$\ExtVer \nleqW \pitaccN$.
\begin{proof}
Since $\pitaccN \nleqW \mflim$ (\thref{prop:pitacc_not_below_lim}), it suffices (by \thref{prop:pi02accnlowerbound}) to prove that if $\ExtVer \leqW \pitaccN$, then $\ExtVer \leqW \mflim$. We will in fact reduce $\ExtVer$ to $\lpo$. 

Observe that the forward functional of a putative reduction $\operatorname{ExtVer} \leqW \pitaccN$ constitutes a computable procedure for producing a sequence of vertices $(v_i)_{i \in \mathbb{N}}$ in a given tree of infinite width such that there is at most one vertex $v_i$ which is not extendible (to an infinite antichain). To reduce $\ExtVer$ to $\lpo$, ask $\lpo$ if there are $j, k$ with $v_j \pprefix v_k$. If so, then $v_k$ must be extendible because (1) at most one $v_i$ is not extendible; (2) if $v_j$ is extendible, so is $v_k$. On the other hand, if $(v_i)_{i \in \mathbb{N}}$ is already an infinite antichain, then $v_0$ is extendible.
\end{proof}
\end{proposition}

To summarize:

\begin{corollary}
\thlabel{corr:extver}
$\pitaccN \leW \ExtVer \equivW \fop{\BStree} \leW \bigsqcap \limits_{k \geq 2} \pitacc{k}$.
\end{corollary}

\begin{corollary}
	$\ACC_\mathbb{N} \times \ACC_\mathbb{N}\not\weireducible \BStree$. In particular, $\fop{\BStree} \leW \fop{\BStree} \times \fop{\BStree}$ and $\BStree \leW \BStree \times \BStree$.
\begin{proof}
	By \thref{prop:bstreefirstorderupperbound}, $\ACC_\mathbb{N} \times \ACC_\mathbb{N}\weireducible \BStree$ implies $\ACC_\mathbb{N} \times \ACC_\mathbb{N} \leqW \pitacc{2} \weiequiv \RT{1}{2}$, which is easily seen to be false. The rest of the statement follows easily from the fact that $\BStree \times \BStree\weireducible \BStree$ implies $\fop{\BStree} \times \fop{\BStree}\weireducible \fop{\BStree}$ (see e.g.\ \cite[Proposition 4.1]{soldavalenti}), which, in turn, implies $\ACC_\mathbb{N} \times \ACC_\mathbb{N} \weireducible \BStree$.
\end{proof}
\end{corollary}

We end this section by drawing attention to the following open question:
\begin{question}
Does $\BStree \leqW \DS$?
\end{question}

All lower bounds for $\BStree$ known to us are also below $\DS$; this in particular applies to the finitary, deterministic, and first-order parts of $\BStree$. On the other hand, the best upper bound for $\BStree$ we have is $\parallelization{\pitacc{k}}$ (\thref{prop:bswtrupperbound}), which in turn does not reduce to $\DS$ (or even $\BS$) by \thref{cor:ACC_DS_lim}. One natural attempt to reduce $\BStree$ to $\DS$ is blocked by the following observation:

\begin{proposition}
	There is no computable procedure that receives as input any tree $T \subseteq \cantor$ and returns a linear extension $(T,\prec)$ of the prefix relation on $T$ such that whenever $T$ has infinite width, then $\prec$ is ill-founded.
\end{proposition}
\begin{proof}
	For the sake of a contradiction, assume there exists such a procedure. We shall define a pruned binary tree $T$ of infinite width on which the procedure produces a well-ordering. At each stage, we define one new level of $T$ as follows. Start by putting the root, $0$, and $1$ into $T$. Wait for the procedure to decide whether $0 \prec 1$ or $1 \prec 0$. While waiting, we extend $0$ and $1$ by $0$s. Eventually, say at stage $s_0$, the procedure decides $w_0 \in \{0,1\}$ is the $\prec$-max of $0$ and $1$. We then add a split above $w_0$, i.e., put $w_0 \concat 0^{s_0}0$ and $w_0\concat 0^{s_0}1$ into $T$. In general, if we added a split into $T$ at stage $s_n$, let $A_n$ denote the $T$-antichain of vertices which are currently $T$-maximal. We wait for the procedure to decide $\max_\prec A_n$. While waiting, we extend each $T$-maximal element by $0$s. Once $w_n$ has been decided as $\max_\prec A_n$, we add a split above $w_n\concat 0^k$ where $k$ is largest such that $w_n\concat 0^k \in T$. We end the stage by extending all other $T$-maximal elements by $0$.
	
	The construction adds infinitely many splits, so $T$ has infinite width. To show that $(T,\prec)$ is a well-ordering, first observe that $(w_n)_n$ forms an increasing cofinal sequence in $(T,\prec)$. It therefore suffices to show that each $w_n$ lies in the $\prec$-well-founded part $W$. Suppose towards a contradiction that there is some infinite $\prec$-descending sequence $S$ which begins with $w_n$. There must be some $a \in A_n$, different from $w_n$, such that infinitely many elements of $S$ are $T$-above $a$. This is because if $x\prec b$ for every $b\in A_n$, then $x$ is a prefix of some of them. Then the subtree of $T$ above $a$ must contain splits; furthermore, some $s \in S$ lies $T$-above such a split (otherwise the subsequence of $S$ which is $T$-above $a$ forms a $T$-chain, which must have $\prec$-order type $\omega$). However, by construction of $T$, such a split occurs above some $w_m$ where $m > n$, so we have $s \succ w_m \succ w_n$. This contradicts the fact that $S$ begins with $w_n$.
\end{proof}

\section{The quotient relative to the parallel product}
\label{sec:quotient}

\thref{theo:dsproducts} implies that
\[ \max_{\leqW} \{h \st h\times \parallelization{\ACC_\mathbb{N}} \leqW \BS\} \]
exists and is equal to (the Weihrauch degree of) $\mflim$.
The existence of such a maximum (with different problems in place of $\parallelization{\ACC_\mathbb{N}}$ and $\BS$) was discussed in \cite[Theorem 3.7]{dghpp20},
which prompted them to ask about the extent to which the ``parallel quotient'' operator
\[ f/g \equivW \max_{\leqW} \{h \st h \times g \leqW f\} \]
is defined.
While \cite[Remark 3.11]{dghpp20} says ``We have no reason to think this operator is total [..]'',
we address their question by showing that the operator in question is in fact total.

\begin{definition}
\thlabel{def:quotient}
Given $f,g\pmfunction{\Baire}{\Baire}$ with $g$ different from $0$, we define their \emph{parallel quotient} $f/g\pmfunction{\mathbb{N}\times\mathbb{N}\times\Baire}{\Baire}$ as follows:
\begin{align*}
\dom(f/g) := \{(e,i,p)\st (\forall q \in \dom(g))[ &\Phi_e(p,q) \in \dom(f) \text{ and } \\
&(\forall r \in f(\Phi_e(p,q)))(\Phi_i(p,q,r) \in g(q))]\}
\end{align*}
\[ f/g(e,i,p) := \{\pairing{q,r}\st q \in \dom(g) \ \land \ r \in f(\Phi_e(p,q))\}\]

\end{definition}

\begin{proposition}
	\thlabel{thm:quotient_total}
$f/g \equivW \max_{\leqW} \{h\st h \times g \leqW f\}$.
\end{proposition}
\begin{proof}
Given a $g$-instance $q$ and an $f/g$-instance $(e,i,p)$,
consider the $f$-instance $\Phi_e(p,q)$.
Given an $r \in f(\Phi_e(p,q))$, we may compute (using $q$ and $e,i,p$) $\Phi_i(p,q,r)$, which is a $g$-solution of $q$,
and $\pairing{q,r}$, which is an $f/g$-solution of $(e,i,p)$.
This proves $f/g \times g \leqW f$.

Suppose $h$ is a problem such that $h \times g \leqW f$ via $\Phi_e$, $\Psi$. We want to show that $h \leqW f/g$. Let $p$ be an $h$-instance. Observe that, in order to compute a solution for $h(p)$, we only need a $g$-input $q$ and any solution $r\in f(\Phi_e(p,q))$. 

Let $i$ be an index for the functional which takes in input $(p,q,r)$ and produces the projection of $\Psi(p, q, r)$ to the second coordinate.
Given the $h$-instance $p$, we uniformly compute the $f/g$-instance $(e,i,p)$.
For every $\pairing{q,r}\in f/g(e,i,p)$, the projection of $\Psi(p,q,r)$ to the first coordinate is an $h$-solution to $p$.
\end{proof}

It follows that $f/g$ is well-defined on all (non-zero) Weihrauch degrees, so we may extend \thref{def:quotient} to all (non-zero) problems. The following proposition is a straightforward consequence of the definitions.

\begin{proposition}
\begin{enumerate}
    \item If $g$ is pointed, then $f/g \weireducible f$.
    \item If $f_0 \weireducible f_1$ and $g_0 \weireducible g_1$, then $f_0/g_1 \weireducible f_1/g_0$.
\end{enumerate}
\end{proposition}

The parallel quotient is useful to describe to what extent a problem $f$ is stable under parallel product. If $f$ is pointed and closed under product, then $f/f \weiequiv f$. At the same time, if $f$ is not closed under product, the quotient $f/f$ gives a quantitative estimate of ``the lack of closure of $f$ under parallel product''.

The parallel quotient is an example of a \emph{residual operator}. In general, for a fixed operator~$\cdot$, it is interesting to investigate the existence of the following degrees: $\max \{h \st f \cdot h \weireducible g\}$, $\max \{h \st h \cdot f \weireducible g\}$, $\min \{h \st g \weireducible f \cdot h \}$, and $\min \{h \st g \weireducible h \cdot f\}$. A residual operator for~$\cdot$ is one that witnesses the existence of one (or more) of these maxima/minima. The existence of residual operators for the join $\sqcup$ and the meet $\sqcap$ in the Weihrauch degrees has been studied in \cite{HiguchiPauly13}. Besides, the existence of $\min \{h \st g \weireducible f \compproduct h\}$ has been proven in \cite{paulybrattka4} (resulting in the \emph{implication} operator $\to$). The parallel quotient is a residual operator for $\times$. A more thorough study of the existence of residual operators for other choices of $\cdot$ is currently under investigation. Further properties of the parallel quotient will be discussed in an upcoming paper. 

As already observed, \thref{theo:dsproducts} can be rewritten as 
\[ \DS/\parallelization{\ACC_\mathbb{N}}\weiequiv \BS/\parallelization{\ACC_\mathbb{N}} \weiequiv \mflim. \]

We now discuss the parallel quotient of $\DS$ and $\BS$ over other well-known problems. 

\begin{proposition}
	\thlabel{thm:DSxCN=DS}
	$\DS \weiequiv \CNatural \times \DS \equivW \DS \compproduct \CNatural$ and $\BS \weiequiv \CNatural \times \BS \equivW \BS \compproduct \CNatural$.
\end{proposition}

\begin{proof}
	The reductions $\DS \weireducible \CNatural \times \DS \weireducible \DS \compproduct \CNatural$ are trivial. To show that $\DS\compproduct \CNatural \weireducible \DS$, let $(w,p)\in\dom(\DS\compproduct \CNatural)$. Without loss of generality, we may view the (possibly finite) string produced by $\Phi_w(n)$ as a linear order $L_n$. Clearly $L_n$ need not be ill-founded if $n\notin \CNatural(p)$.
	
	We shall uniformly compute a linear order $L$ as follows: the elements of $L$ will be of the form $(n,x,s)$, where 
	\[ (n,x,s) \le_L (m,y,s') \defiff n<m \text{ or } (n=m \text{ and } x\le_{L_n} y).\] 
	The third component plays no role in the order and is only used to guarantee that the resulting linear order is computable (and not just c.e.). The construction below will guarantee that no two triples share the same third component.  

	Starting from $n = 0$, as long as $n$ is not enumerated outside of $\CNatural(p)$, we build an isomorphic copy of $L_n$ at the top of $L$. More precisely we add $(n,x,s)$ to $L$, where $s$ is the current step of the construction and $x$ is the $\mathbb{N}$-least element of $L_n$ which has not been added.

    If we ever see $n \notin \CNatural(p)$, we continue the above with $n+1$ instead of $n$.

    This completes the construction of $L$.	Observe that if $n = \min \CNatural(p)$ then, modulo a finite initial segment, $L$ is isomorphic to $L_n$. In particular, $L$ is ill-founded and, given any infinite descending sequence through $L$, we can uniformly compute $n$ and a descending sequence through $L_n$. 

	The very same argument works for $\BS$ in place of $\DS$: just replace ``linear order'', ``ill-founded'', and ``descending sequence'' with ``quasi-order'', ``a non-well quasi order'' and ``bad sequence'' respectively.
\end{proof}

\begin{lemma}
	\thlabel{thm:f_over_fo_problem}
	Let $g$ be computably true (i.e., every $g$-instance has a $g$-computable solution). Suppose $g\leqW f\restrict{X}$, where $f\restrict{X}$ is such that, for every $x\in X$ and every $y\in f(x)$, $y\not\turingreducible x$. Then $f/g\not\leqW g$.
\end{lemma}
\begin{proof}
    Let $(e,i,p)\in\dom(f/g)$ be such that $p$ is computable and the maps $q\mapsto \Phi_e(p, q)$ and $(q,r)\mapsto \Phi_i(p,q,r)$ witness the reduction $g\leqW f\restrict{X}$. Observe that, by definition of $f/g$ and $(e,i,p)$, every solution of $f/g(e,i,p)$ computes a solution for $f(\Phi_e(p,q))$ for some input $q\in\dom(g)$. In particular, since $\Phi_e(p,q)\in X$, $f/g(e,i,p)$ does not have any computable solution.
	
	If $f/g \leqW g$ as witnessed by $\Phi,\Psi$, then for every $t\in g\Phi(e,i,p)$, $\Psi( (e,i,p), t) \in f/g(e,i,p)$. This is a contradiction as $g$ is computably true and therefore $g\Phi(e,i,p)$ always contains a computable point.
\end{proof}

\begin{proposition}
	For $f\in\{\DS,\BS\}$, $\PiBound\leW  f / \PiBound\leqW f / \RT{1}{2} \leW f$.
\end{proposition}
\begin{proof}
	The reduction $\PiBound\leqW  \DS / \PiBound$ follows from the fact that $\PiBound$ is closed under product (see e.g.\ \cite[Theorem 7.16]{soldavalenti}) and $\PiBound \weireducible \DS$. To show that this reduction is strict, we shall use \thref{thm:f_over_fo_problem}. Let $X$ be the set of ill-founded linear orders $L$ with no $L$-computable descending sequence. One can modify the reduction \cite[Proposition 4.8]{goh-pauly-valenti} from $\PiBound$ to $\DS$ to show
    $\PiBound \leqW \DS|_X$: Given a $\PiBound$-instance (thought of as a sequence of trees $(T_n)_n$), define a $(T_n)_n$-computable ill-founded tree $S$ with no $(T_n)_n$-computable path (this can be done in a manner which is uniformly computable from $(T_n)_n$). Then define a linear order $L = \bigcup_n \{n\} \times \mathrm{KB}(T_n \times S)$. Note that if $T_n$ is well-founded, then so is $T_n \times S$, while if $T_n$ is ill-founded, then $T_n \times S$ is ill-founded with no $(T_n)_n$-computable path. So $L$ is ill-founded with no $L$-computable descending sequence. The backward functional works as in \cite[Proposition 4.8]{goh-pauly-valenti}. It follows from \thref{thm:f_over_fo_problem} that $\DS/\PiBound \nleqW \PiBound$.

	The reduction $\DS / \PiBound\leqW \DS / \RT{1}{2}$ follows by the antimonotonicity of the quotient operator. The reduction $ \DS / \RT{1}{2} \leqW \DS$ is immediate as $\RT{1}{2}$ is pointed. The fact that it is strict follows from the fact that if $\DS \leqW \DS /\RT{1}{2}$ then, by definition, $\RT{1}{2}\times \DS \weireducible \DS$. This is impossible as, by \thref{theo:dsproducts}, $\RT{1}{2}\times \parallelization{\ACC_\mathbb{N}}\not\weireducible \DS$.

	Replacing $\DS$ with $\BS$ in the above argument shows that the statement holds for $\BS$ as well.
\end{proof}

Notice that, in the proof of the previous result, we used the fact that both $\DS/\PiBound$ and $\BS/\PiBound$ (and hence $\DS/\RT{1}{2}$ and $\BS/\RT{1}{2}$) have computable instances with no computable solutions. In particular, this implies that none of them is a first-order problem. The following result provides a lower bound for $\DS/\RT{1}{2}$ and $\BS/\RT{1}{2}$.

Let us denote with $\DSfe$ (resp.\ $\BSfe$) the problem ``given an ill-founded linear order (resp.\ non-well quasi-order) $X$, compute a sequence that, cofinitely, is a $X$-descending sequence (resp.\ $X$-bad sequence)''. We mention that $\DSfe$ and $\BSfe$ can be characterized as $\CNatural \rightarrow \DS$ and $\CNatural \rightarrow \BS$ respectively.
\begin{proposition}
	\thlabel{thm:RT12xDSfe<=DS}
	$\RT{1}{2} \times \DSfe \leqW \DS$ and $\RT{1}{2} \times \BSfe \leqW \BS$.
\end{proposition}
\begin{proof}
	We only prove that $\RT{1}{2} \times \DSfe \leqW \DS$; the proof for $\BS$ can be obtained analogously by adapting the following argument. 
	Let $(c,L)$ be an instance of $\RT{1}{2} \times \DSfe$, where $c\function{\mathbb{N}}{2}$ is a $2$-coloring and $L=\sequence{x_n}{n\in\mathbb{N}}$ is an ill-founded linear order. We build a linear order $M$ in stages. Intuitively, we want to build a copy of $L$ until we see that the current color changes. When this happens, we start building a fresh copy of $L$ at the top of the current linear order. However, to guarantee that the resulting order will be ill-founded, we also build another copy of $L$ at the bottom, which gets expanded every time two color changes are observed.

	Formally, we construct the order $M$ as follows: elements are of the form $(n,x,b,s)$ and are ordered as 
	\[ (n,x,b,s) \le_M (m,y,b',s') \defiff n<m \text{ or } (n=m \text{ and } x\le_L y).  \]
	The third and fourth components play no role in the order $M$. The construction below will guarantee that this is a linear order (rather than a quasi-order). 
	
	At stage $0$, we add $(1,x_0,c(0),0)$ to $M$. We also define $p_0:=-1$, $q_0:=0$, and $h_0:=1$. Intuitively, $p_s$ stores the last position in the bottom copy of $L$ and $q_s$ stores the last position in the top copy of $L$.

	At stage $s+1$, we distinguish the following cases: 
	\begin{itemize}
		\item if $c(s+1)=c(s)=b$, we add $(h_s,x_{q_s+1},b, s+1)$ to $M$. We define $q_{s+1}:=q_s+1$, $p_{s+1}:=p_s$, $h_{s+1}:=h_s$ and go to the next stage.
		\item if $c(s+1)=1$ and $c(s)=0$, we add $(h_s+1,x_{0},1,s+1)$ to $M$. We define $q_{s+1}:=0$, $p_{s+1}:=p_s$, $h_{s+1}:=h_s+1$ and go to the next stage.
		\item if $c(s+1)=0$ and $c(s)=1$, we add $(h_s+1,x_{0},0, s+1)$ to $M$. We also add $(0,x_{p_s+1},2, s+1)$ to $M$. We define $q_{s+1}:=0$, $p_{s+1}:=p_s+1$, $h_{s+1}:=h_s+1$ and go to the next stage.
	\end{itemize}
	This concludes the construction.

	Observe that if the coloring $c$ is eventually constant and $\lim_n c(n)=b$, then $M \cong N+L$ for some $N\in\mathbb{N}$ (in particular it is ill-founded) and every descending sequence through $M$ is of the form $\sequence{(n,y_i,b,s_i)}{i\in\mathbb{N}}$ for some $n$, where $\sequence{y_i}{i\in\mathbb{N}}$ is a $L$-descending sequence. In particular, we can compute a solution for $(\RT{1}{2} \times \DSfe)(c,L)$ via projections. 

	On the other hand, if $c$ is not eventually constant, then $M \cong L + \omega$. Every descending sequence through $M$ must eventually list elements of the initial segment of $M$ isomorphic to $L$. In particular, if $\sequence{(n_i,y_i,b_i,s_i)}{i\in\mathbb{N}}$ is an $M$-descending sequence then $b_0\in \RT{1}{2}(c)$ (trivially) and $\sequence{y_i}{i\in\mathbb{N}} \in \DSfe(L)$.
\end{proof}

Let $\Sort_2: \Cantor \to \Cantor$ be defined by $\Sort_2(p) = 0^i1^\omega$ if there are exactly $i$-many $n$ such that $p(n)=0$, otherwise $\Sort_2(p) = 0^\omega$. While $\Sort_2$ and $\RT{1}{2}$ are quite weak on their own, their parallel product does not reduce to $\BS$:

\begin{proposition}
	\thlabel{thm:sort2xRT12_not_le_BS}
	$\Sort_2 \times \RT{1}{2}\not\weireducible \BS$.
\end{proposition}
\begin{proof}
	Observe first of all that $\BS$ is a cylinder, i.e., $\id \times \BS \leqsW \BS$ (this can be easily proved by adapting the proof of \cite[Proposition 4.6]{goh-pauly-valenti}). This means that it suffices to prove $\Sort_2 \times \RT{1}{2} \not\leqsW \BS$ (see \cite[11.3.5]{bgp21}).
    
	Assume towards a contradiction that the functionals $\Phi,\Psi_0,\Psi_1$ witness the strong reduction $\Sort_2 \times \RT{1}{2}\strongweireducible \BS$ (where $\Psi_0$ and $\Psi_1$ produce, respectively, the $\Sort_2$ and the $\RT{1}{2}$ answers). We build an instance $(p,c)$ of $\Sort_2 \times \RT{1}{2}$ in stages as follows. At each stage $s$, we have a finite string $\rho_s$, a finite coloring $\gamma_s$, and a set $D_s$ of finite bad sequences in the quasi-order $\preceq_s$ produced by $\Phi(\rho_s,\gamma_s)$ in $s$ steps. Intuitively, the set $D_s$ corresponds to the set of finite bad sequences that have already been forced to be non-extendible. Recall the quasi-order $\trianglelefteq$ (\thref{def:bad_seq_quasiorder}) on bad sequences.

	We start with $\rho_0 := \gamma_0 := \varepsilon$ and $D_0:=\emptyset$. At stage $s+1=2n+1$, we consider the set $M_{s+1}$ of all the $\trianglelefteq$-maximal bad sequences $\beta$ in $\preceq_{s+1}$ that are not in $D_s$ and such that 
	\begin{itemize}
		\item $\Psi_1(\beta)(0)\downarrow=k<2$
		\item $\Psi_0(\beta)$ did not commit to a finite number of zeroes for the $\Sort_2$ solution (i.e.\ the string produced by $\Psi_0(\beta)$ in $s+1$ steps is not of the form $0^i\concat 1 \sigma$ for any $i$ or $\sigma$).
	\end{itemize}
	For each such $\beta$, we search for a sufficiently large $m\in\mathbb{N}$ such that there exists a finite bad sequence $\alpha$ in $\Phi(\rho_s \concat 1^m, \gamma_s \concat (1-k)^m)$ such that $\beta \trianglelefteq \alpha$ and there is $i$ such that $0^i1 \prefix \Psi_0(\alpha)$ within $s+1+m$ steps. In other words, we extend our finite input of $\Sort_2$ and finite coloring $\gamma_s$ so as to force the forward functional to produce some finite bad sequence $\alpha \trianglerighteq \beta$ such that $\Psi_0(\alpha)$ commits to a finite number of zeroes for the $\Sort_2$ solution. Observe that this needs to happen because $\beta$ cannot be extendible in $\Phi(\rho_s\concat 1^\omega, \gamma_s\concat (1-k)^\omega)$ (it is producing the wrong $\RT{1}{2}$ answer) and hence, by \thref{lem}, there is a finite bad sequence $\alpha$ in $\Phi(\rho_s\concat 1^\omega, \gamma_s\concat (1-k)^\omega)$ as desired. We diagonalize against $\alpha$ by extending $\rho_s\concat 1^m$ with $(i+1)$-many zeroes. We then add $\alpha$ and every bad sequence $\trianglelefteq$-below $\alpha$ to $D_s$. This concludes the action needed for $\beta$. Once all the bad sequences in $M_{s+1}$ have received attention, we go to the next stage. 

	At stage $s+1=2n+2$, we consider the set $M_{s+1}$ of all the $\trianglelefteq$-maximal bad sequences $\beta$ in $\preceq_{s+1}$ that are not in $D_s$ and such that, for some $i$, $0^i1 \prefix \Psi_0(\beta)$. Let 
	\[j:= \max \{ i \st (\exists \beta \in M_{s+1})(0^i1\prefix \Psi_0(\beta)) \}+1.\]
	We can uniformly compute such $j$ because $M_{s+1}$ is finite. We define $\rho_{s+1}:=\rho_s\concat 0^j1$ and $\gamma_{s+1}:=\gamma_s\concat 0$. We also add the $\trianglelefteq$-downward closure of $M_{s+1}$ into $D_s$ and go to the next stage. 

	This concludes the construction. It is apparent that the strings $\rho_s$ and $\gamma_s$ are extended infinitely often, and therefore the pair $(p,c)$ with $p:= \bigcup_s \rho_s$ and $c:=\bigcup_s \gamma_s$ is a valid instance for $\Sort_2\times \RT{1}{2}$. Let $\preceq$ be the non-well quasi-order defined by $\Phi(p,c)$ and let $B\in \BS(\Phi(p,c))$. By continuity, let $n$ be such that $\Psi_1(B[n])\downarrow$ and let $s+1$ be the first stage by which every element of $B[n]$ enters $\preceq$. We show that there is a finite bad sequence $\alpha\trianglerighteq B[n]$ which produces a wrong answer for $\Sort_2\times \RT{1}{2}(p,c)$. This suffices to reach a contradiction, as, by \thref{lem}, $B[n]$ being extendible implies that $\alpha$ is extendible.
	
	Assume first that $\Psi_0(B[n])_{s+1}$ commits to finitely many zeroes. Let $t\ge s+1$ be even. Let $\alpha \trianglerighteq B[n]$ be $\trianglelefteq$-maximal in $\preceq_{t}$ such that, for some $i$, $0^i1\prefix\Psi_0(\alpha)_{t}$. By construction, at stage $t$ we are adding more than $i$-many zeroes in $p$, therefore $\Psi_0(\alpha)$ is not the prefix of $\Sort_2(p)$.

	If $\Psi_0(B[n])_{s+1}$ does not commit to finitely many zeroes, let $r\ge s+1$ be odd. Let $\beta \trianglerighteq B[n]$ be $\trianglelefteq$-maximal in $\preceq_{r}$ such that $\beta\in M_r$. At stage $r$, we forced the forward functional $\Phi$ to produce some $\alpha \trianglerighteq \beta$ such that, for some $i$, $0^i1\prefix\Psi_0(\alpha)$, and then extended $\rho_r$ by adding more than $i$-many zeroes. In other words,  $\Psi_0(\alpha)$ is not a prefix of $\Sort_2(p)$, and this concludes the proof.
\end{proof}

Observe that $\Sort_2$ is one of the weakest ``natural'' problems that require infinitely many mind-changes to be solved. The previous result shows that, roughly speaking, while the quotient of $\DS$ (resp.\ $\BS$) over $\RT{1}{2}$ is capable of solving every problem that can be solved with finitely-many mind changes (essentially because $\RT{1}{2}\times \RT{1}{\mathbb{N}}\weiequiv \RT{1}{\mathbb{N}}\weireducible \DS$, see \cite[Proposition 4.24]{goh-pauly-valenti}), allowing infinitely-many mind changes constitutes a critical obstacle. This leads to the following open problem:

\begin{question}
	Characterize $\DS/\RT{1}{2}$ and $\BS/\RT{1}{2}$.
\end{question}
	
In an attempt to better understand the stability of $\DS$ and $\BS$ under parallel products, we highlight also the following open question:

\begin{question}
	What are the degrees of $\DS/\DS$ and $\BS/\BS$?
\end{question}

Observe that, as a consequence of \thref{thm:DSxCN=DS} and \thref{thm:sort2xRT12_not_le_BS}, $\CNatural\weireducible \DS/\DS \leW \mflim$. It is therefore natural to ask whether $\DS/\DS \weireducible \CNatural$. A possible way to separate $\DS/\DS$ and $\CNatural$ would be showing that $\DS$ is equivalent to its restriction to inputs $L$ with no $L$-computable solution. This can be rephrased as asking whether $\NON\times \DS \weireducible \DS$, where $\NON(p):=\{q\in \Baire \st q\not\turingreducible p\}$. The same observations can be made by substituting $\DS$ with $\BS$.

\bibliographystyle{mbibstyle}
\bibliography{references}

\providecommand{\bysame}{\leavevmode\hbox to3em{\hrulefill}\thinspace}
\providecommand{\MR}{\relax\ifhmode\unskip\space\fi MR }
\providecommand{\MRhref}[2]{%
  \href{http://www.ams.org/mathscinet-getitem?mr=#1}{#2}
}
\providecommand{\href}[2]{#2}
\begin{thebibliography}{10}

\bibitem{brattka_godel}
Brattka,  Vasco, \emph{On the complexity of learning programs}, Unity of logic
  and computation (Della~Vedova,  Gianluca, Dundua,  Besik, Lempp,  Steffen,
  and Manea,  Florin, eds.), Lecture Notes in Computer Science, vol. 13967,
  Springer, Cham, 2023, \doi{10.1007/978-3-031-36978-0\_14}, pp.~166--177.
  \MR{4638155}

\bibitem{BolWei11}
Brattka,  Vasco, Gherardi,  Guido, and Marcone,  Alberto, \emph{The
  {B}olzano-{W}eierstrass {T}heorem is the jump of {W}eak {K}{\"o}nig's
  {L}emma}, Annals of Pure and Applied Logic \textbf{163} (2012), no.~6,
  623--655, \doi{10.1016/j.apal.2011.10.006}. \MR{3650357}

\bibitem{bgp21}
Brattka,  Vasco, Gherardi,  Guido, and Pauly,  Arno, \emph{Weihrauch Complexity
  in Computable Analysis}, Handbook of Computability and Complexity in Analysis
  (Brattka,  Vasco and Hertling,  Peter, eds.), Springer International
  Publishing, Jul 2021, \doi{10.1007/978-3-030-59234-9_11}, pp.~367--417.
  \MR{4300761}

\bibitem{paulybrattka4}
Brattka,  Vasco and Pauly,  Arno, \emph{On the algebraic structure of Weihrauch
  degrees}, Logical Methods in Computer Science \textbf{14} (2018), no.~4,
  1--36, \doi{10.23638/LMCS-14(4:4)2018}. \MR{3868998}

\bibitem{BRramsey17}
Brattka,  Vasco and Rakotoniaina,  Tahina, \emph{On the uniform computational
  content of Ramsey's theorem}, The Journal of Symbolic Logic \textbf{82}
  (2017), no.~4, 1278--1316, \doi{10.1017/jsl.2017.43}. \MR{3743611}

\bibitem{calvertfranklinturetsky}
Calvert,  Wesley, Franklin,  Johanna N.~Y., and Turetsky,  Dan,
  \emph{Structural highness notions}, The Journal of Symbolic Logic \textbf{88}
  (2023), no.~4, 1692--1724, \doi{10.1017/jsl.2022.35}. \MR{4679250}

\bibitem{carluccimainardizdanowski}
Carlucci,  Lorenzo, Mainardi,  Leonardo, and Zdanowski,  Konrad,
  \emph{Reductions of well-ordering principles to combinatorial theorems},
  preprint, 2024, available at \url{https://arxiv.org/abs/2401.04451}.

\bibitem{CHM09}
Chubb,  Jennifer, Hirst,  Jeffry~L., and McNicholl,  Timothy~H., \emph{Reverse
  mathematics, computability, and partitions of trees}, The Journal of Symbolic
  Logic \textbf{74} (2009), no.~1, 201--215, \doi{10.2178/jsl/1231082309}.
  \MR{2499427}

\bibitem{CMVCantorBendixson}
Cipriani,  Vittorio, Marcone,  Alberto, and Valenti,  Manlio, \emph{The
  Weihrauch lattice at the level of $\boldsymbol{\Pi}^1_1\mathsf{-CA}_0$: the
  Cantor-Bendixson theorem}, The Journal of Symbolic Logic (2025), 1--39,
  \doi{10.1017/jsl.2024.72}, published online.

\bibitem{paulycipriani1}
Cipriani,  Vittorio and Pauly,  Arno, \emph{Embeddability of graphs and
  {W}eihrauch degrees}, preprint, 2023, available at
  \url{https://arxiv.org/abs/2305.00935}.

\bibitem{Downey98}
Downey,  Rod~G., \emph{Computability theory and linear orderings}, Handbook of
  Recursive Mathematics, Vol.\ 2 (Ershov,  Yu.~L., Goncharov,  S.S., Nerode,
  A., Remmel,  J.B., and Marek,  V.W., eds.), Studies in Logic and the
  Foundations of Mathematics, vol. 139, North-Holland, Amsterdam, 1998,
  \doi{10.1016/S0049-237X(98)80047-5}, pp.~823--976. \MR{1673590}

\bibitem{dghpp20}
Dzhafarov,  Damir~D., Goh,  Jun~Le, Hirschfeldt,  Denis~R., Patey,  Ludovic,
  and Pauly,  Arno, \emph{Ramsey's theorem and products in the {W}eihrauch
  degrees}, Computability \textbf{9} (2020), no.~2, 85--110,
  \doi{10.3233/com-180203}. \MR{4100139}

\bibitem{DSY23}
Dzhafarov,  Damir~D., Solomon,  Reed, and Yokoyama,  Keita, \emph{On the
  first-order parts of problems in the Weihrauch degrees}, Computability
  \textbf{13} (2024), no.~3-4, 363--375, \doi{10.3233/COM-230446}.

\bibitem{freund2023logical}
Freund,  Anton, Pakhomov,  Fedor, and Sold\`a,  Giovanni, \emph{The logical
  strength of minimal bad arrays}, Proceedings of the American Mathematical
  Society \textbf{152} (2024), no.~12, 4993--5005, \doi{10.1090/proc/17038}.
  \MR{4856392}

\bibitem{goh-pauly-valenti}
Goh,  Jun~Le, Pauly,  Arno, and Valenti,  Manlio, \emph{Finding descending
  sequences through ill-founded linear orders}, The Journal of Symbolic Logic
  \textbf{86} (2021), no.~2, 817--854, \doi{10.1017/jsl.2021.15}, arxiv:
  \url{https://arxiv.org/abs/2010.03840}. \MR{4328030}

\bibitem{gohpaulyvalenti2-cie}
\bysame, \emph{The Weakness of Finding Descending Sequences in Ill-Founded
  Linear Orders}, Twenty Years of Theoretical and Practical Synergies (Cham)
  (Levy~Patey,  Ludovic, Pimentel,  Elaine, Galeotti,  Lorenzo, and Manea,
  Florin, eds.), Lecture Notes in Computer Science, Springer Nature
  Switzerland, 2024, \doi{10.1007/978-3-031-64309-5\_27}, pp.~339--350.

\bibitem{HiguchiPauly13}
Higuchi,  Kojiro and Pauly,  Arno, \emph{The degree structure of Weihrauch
  reducibility}, Logical Methods in Computer Science \textbf{9} (2013),
  no.~2:02, 1--17, \doi{10.2168/LMCS-9(2:02)2013}. \MR{3045629}

\bibitem{HJ16}
Hirschfeldt,  Denis~R. and Jockusch,  Carl~G., Jr., \emph{On notions of
  computability-theoretic reduction between {$\Pi^1_2$} principles}, Journal of
  Mathematical Logic \textbf{16} (2016), no.~1, 1650002(1--59),
  \doi{10.1142/s0219061316500021}. \MR{3518779}

\bibitem{kmp20}
Kihara,  Takayuki, Marcone,  Alberto, and Pauly,  Arno, \emph{Searching for an
  analogue of {$ATR_0$} in the Weihrauch lattice}, The Journal of Symbolic
  Logic \textbf{85} (2020), no.~3, 1006--1043, \doi{10.1017/jsl.2020.12}.
  \MR{4231614}

\bibitem{paulyleroux}
Le~Roux,  St\'ephane and Pauly,  Arno, \emph{Finite choice, convex choice and
  finding roots}, Logical Methods in Computer Science \textbf{11} (2015),
  no.~4, 4:6, 30, \doi{10.2168/LMCS-11(4:6)2015}. \MR{3430498}

\bibitem{LMV24}
Lempp,  Steffen, Marcone,  Alberto, and Valenti,  Manlio, \emph{Chains and
  antichains in the Weihrauch lattice}, preprint, 2024, available at
  \url{https://arxiv.org/abs/2411.07792}.

\bibitem{MarconeWQO}
Marcone,  Alberto, \emph{WQO and BQO theory in subsystems of second order
  arithmetic}, Lecture Notes in Logic, pp.~303--330, Cambridge University
  Press, 2005, \doi{10.1017/9781316755846.020}. \MR{2185443}

\bibitem{NP18}
Neumann,  Eike and Pauly,  Arno, \emph{A topological view on algebraic
  computation models}, Journal of Complexity \textbf{44} (2018), 1--22,
  \doi{10.1016/j.jco.2017.08.003}. \MR{3724808}

\bibitem{paulypradicsolda}
Pauly,  Arno, Pradic,  Cecilia, and Sold\`a,  Giovanni, \emph{On the
  {W}eihrauch degree of the additive {R}amsey theorem}, Computability
  \textbf{13} (2024), 459--483, \doi{10.3233/COM-230437}.

\bibitem{paulysolda}
Pauly,  Arno and Sold{\`a},  Giovanni, \emph{Sequential Discontinuity and
  First-Order Problems}, Twenty Years of Theoretical and Practical Synergies
  (Cham) (Levy~Patey,  Ludovic, Pimentel,  Elaine, Galeotti,  Lorenzo, and
  Manea,  Florin, eds.), Springer Nature Switzerland, 2024,
  \doi{10.1007/978-3-031-64309-5\_28}, pp.~351--365.

\bibitem{SacksHRT}
Sacks,  Gerald~E., \emph{Higher Recursion Theory}, 1 ed., Perspectives in
  Mathematical Logic, Springer-Verlag, Berlin, 1990, \doi{10.1007/BFb0086109}.
  \MR{1080970}

\bibitem{Simpson09}
Simpson,  Stephen~G., \emph{Subsystems of Second Order Arithmetic}, 2 ed.,
  Perspectives in Logic, Cambridge University Press, Cambridge; Association for
  Symbolic Logic, Poughkeepsie, NY, 2009, \doi{10.1017/CBO9780511581007}.
  \MR{2517689}

\bibitem{soldavalenti}
Sold\`a,  Giovanni and Valenti,  Manlio, \emph{Algebraic properties of the
  first-order part of a problem}, Annals of Pure and Applied Logic \textbf{174}
  (2023), no.~7, Paper No. 103270, 41, \doi{10.1016/j.apal.2023.103270}.
  \MR{4583071}

\bibitem{Westrick20diamond}
Westrick,  Linda~Brown, \emph{A note on the diamond operator}, Computability
  \textbf{10} (2021), no.~2, 107--110, \doi{10.3233/COM-20029}. \MR{4246716}

\end{thebibliography}

\end{document}